\documentclass{amsart}

\usepackage{kotex}
\usepackage{amsmath}
\usepackage{amsfonts}
\usepackage{amscd}

\usepackage{amssymb}
\usepackage{graphicx}
\usepackage{dsfont}
\usepackage{color}
\usepackage{hyperref}
\usepackage{bbm}
\usepackage{sseq}
\usepackage{tikz-cd}
\usepackage[abs]{overpic} 
\usepackage{bm}
\usepackage{subcaption}

\usepackage{comment}

\newtheorem{theorem}{Theorem}[section]
\newtheorem{lemma}[theorem]{Lemma}
\newtheorem{conjecture}[theorem]{Conjecture}
\newtheorem{proposition}[theorem]{Proposition}
\newtheorem{corollary}[theorem]{Corollary}

\newtheorem*{question*}{Question}

\theoremstyle{definition}
\newtheorem*{definition*}{Definition}
\newtheorem{definition}[theorem]{Definition}

\newtheorem*{example*}{Example}
\newtheorem{example}[theorem]{Example}
\newtheorem*{observation*}{Observation}

\newtheorem*{Goal*}{Goal}

\newtheorem*{Assumption*}{Assumption}

\theoremstyle{remark}
\newtheorem*{remark*}{Remark}
\newtheorem{remark}[theorem]{Remark}

\numberwithin{equation}{section}

\newcommand{\ow}{\omega}

\newcommand{\lda}{\lambda}

\newcommand{\p}{\partial}

\newcommand{\C}{{\mathbb{C}}}
\newcommand{\R}{{\mathbb{R}}}

\newcommand{\Z}{{\mathbb{Z}}}

\newcommand{\D}{{\mathbb{D}}}

\newcommand{\CO}{\mathcal{CO}}

\DeclareMathOperator{\im}{Im}

\DeclareMathOperator{\Fix}{Fix}

\DeclareMathOperator{\HW}{HW}
\DeclareMathOperator{\CZ}{CZ}
\DeclareMathOperator{\SH}{SH}

\DeclareMathOperator{\HF}{HF}
\DeclareMathOperator{\CF}{CF}
\DeclareMathOperator{\eHW}{eHW}
\DeclareMathOperator{\Ho}{H}

\begin{document}

\title[The systoles of symmetric convex hypersurfaces]{Remarks on the systoles of symmetric convex hypersurfaces and   symplectic capacities}

\author{Joontae Kim, Seongchan Kim, and Myeonggi Kwon}
 \address{School of Mathematics, Korea Institute for Advanced Study, 85 Hoegiro, Dongdaemun-gu, Seoul 02455, Republic of Korea}
 \email {joontae@kias.re.kr}

 \address{Institut de Math\'ematiques, Universit\'e de Neuch\^atel, Rue Emile-Argand 11, 2000 Neuch\^atel, Switzerland}
 \email {seongchan.kim@unine.ch}

  \address{Fakult\"at f\"ur Mathematik, Ruhr-Universit\"at Bochum, Universit\"atsstrasse 150, 44801 Bochum, Germany}
 \email {myeonggi.kwon@rub.de}

%\date{Recent modification; \today}
\begin{abstract}
%In this note, we study the systoles of convex hypersurfaces in $\R^{2n}$ invariant under an anti-symplectic involution. We present an approach to understand the symmetric ratio of the hypersurfaces, i.e. the ratio of the systole to the symmetric systole, using symplectic capacities from Floer theory, and we examine the symmetric ratio with various concrete examples.

In this note we study the systoles of convex hypersurfaces in $\R^{2n}$ invariant under an anti-symplectic involution. We investigate a uniform upper bound of the ratio between the systole and the symmetric systole of the hypersurfaces using symplectic capacities from Floer theory. We discuss various concrete examples in which the ratio can be understood explicitly.

%WHAT ABOUT: In this note we study the systole of a Hamiltonian system on a hypersurface in $\R^{2n}$ that is invariant under an anti-symplectic involution. We present an approach to understand the symmetric systole, i.e. the minimal period of symmetric periodic orbits, compared with the systole for convex hypersurfaces using symplectic capacities from Floer theory. We provide various explicit examples. 

%NEED TO REWRITE? In this note we discuss the problem whether the systole of a Hamiltonian system on a hypersurface in $\R^{2n}$, that is invariant under an anti-symplectic involution, can be realized by a symmetric periodic orbit. We present an approach to this question for convex domains using  symplectic capacities from Floer theory, and we exhibit various explicit examples. %We conjecture that for such a hypersurface, the systole is attained by a symmetric periodic orbit, and we show that this is true up to at most a factor of two. 
\end{abstract}

\maketitle

% \tableofcontents

%\newpage

\section{Introduction}

 Let $ \Sigma $ be a  smooth manifold of dimension $2n-1$ equipped with a global 1-form $\alpha$ such that $\alpha \wedge (d\alpha)^{n-1}$ is nowhere vanishing. Such a   pair $(\Sigma, \alpha)$ is called a (co-oriented) contact manifold. We assume throughout that $\Sigma$ is closed and connected. There exists a unique vector field $R = R_{\alpha}$ on $\Sigma$  characterized by the conditions $d\alpha(R, \cdot) =0$ and $\alpha(R) = 1$. The vector field $R$ is called the Reeb vector field associated with $\alpha$. A periodic (Reeb) orbit is a smooth curve $\gamma  \colon \R/ \tau \Z \to \Sigma$ solving the differential equation $\dot{\gamma} = R \circ \gamma$.   The \emph{systole} of $(\Sigma, \alpha)$ is defined as
\[
 \ell_{\min}(\Sigma, \alpha) = \inf \{ \tau >0  \mid \text{$\tau$ is the period of a periodic orbit on $(\Sigma, \alpha)$}     \} >0.
 \]
 By convention, the infimum of the empty set is infinity.  %{\color{blue} it has been studied intensively in Riemannian geometry and so on... bla bla}

 Suppose that the contact manifold $(\Sigma, \alpha)$ is equipped with an anti-contact involution $\rho$, meaning that $\rho^2 = \mathrm{Id}$ and $\rho^* \alpha  = - \alpha$. The triple   $(\Sigma, \alpha, \rho)$ is called a \emph{real contact manifold}.   A periodic orbit $\gamma $  on $(\Sigma, \alpha, \rho)$ is said to be \emph{symmetric} if $\rho( \im( \gamma )) = \im( \gamma )$. By definition, the \emph{symmetric systole} $\ell_{\min}^{\mathrm{sym}}(\Sigma, \alpha, \rho)$   is the infimum over the periods of symmetric periodic orbits on $\Sigma$. We then define the \emph{symmetric ratio}  $\mathfrak{R} (\Sigma, \alpha, \rho)$  as
\begin{equation}\label{eq: defofratio}
\mathfrak{R} (\Sigma, \alpha, \rho) := \frac{   \ell_{\min}^{\mathrm{sym}}(\Sigma, \alpha, \rho)}{\ell_{\min} (\Sigma, \alpha )} \in [1, \infty],
 \end{equation}
provided that there exists a periodic orbit on $(\Sigma, \alpha)$.

\begin{example}\label{exam:first} Let $\Sigma \subset \R^{2n}$ be a smooth, compact, and starshaped hypersurface with respect to the origin. Assume that $\Sigma$ is invariant under the  complex conjugation $$\rho_0(x_1, y_1, \ldots, x_n, y_n) =  (x_1,- y_1, \ldots, x_n,- y_n). $$  The  triple $(\Sigma, \alpha, \rho_0)$ is a real contact manifold, where $\alpha = \lambda_0 |_{\Sigma}$ is the restriction of the Liouville form $\lambda_0 = \tfrac{1}{2} \sum_{j=1}^n ( x_j dy_j - y_j dx_j)$ to $\Sigma$. The Reeb orbits on $\Sigma$ are reparametrizations of the Hamiltonian orbits on $\Sigma$ of   any Hamiltonian $H \colon \R^{2n} \to \R$ having $\Sigma$ as a regular level set.  The existence of a symmetric periodic orbit was established by Rabinowitz  \cite{Rabisym}, implying that  $ \mathfrak{R} (\Sigma, \alpha, \rho_0) $ is finite.  
 \end{example}
 
As of a prominent class of real contact manifolds, we are mainly interested in symmetric convex hypersurfaces in $\R^{2n}$.
In this paper convex domains in $\R^{2n}$ are assumed to contain the origin in the interior, and starshaped domains are starshaped with respect to the origin.
Let $K \subset \R^{2n}$ be a compact convex domain with smooth boundary which is invariant under an anti-symplectic involution $\rho$ of $\R^{2n}$, i.e.\ $\rho^2=\mathrm{Id}$ and $\rho^* d\lambda_0 = -d\lambda_0$.
We call the boundary of $K$ a \emph{symmetric convex hypersurface}. 
Assume that the fixed point set $\Fix(\rho)$ intersects the boundary $\p K$.
This condition necessarily holds if $(\p K,\rho)$ admits a symmetric periodic orbit.
We can find a Liouville form $\lambda$ on the symplectic manifold $(K,d\lambda_0)$ such that its Liouville vector field is transverse along the boundary $\p K$ and $\rho$ is exact with respect to $\lda$ i.e. $\rho^*\lambda=-\lambda$.
For example one takes the average $\lambda:=\frac{1}{2}(\lambda_0-\rho^*\lambda_0)$, see Lemma \ref{lem: ldaisalsofine}.
We define the \emph{(symmetric) systoles} of the symmetric convex hypersurface $(\p K, \rho)$ by the ones of the real contact manifold $(\p K,\alpha:=\lambda|_{\p K},\rho)$:
\begin{equation}\label{eq: systoleforlda}
\ell_{\min}(\p K):=\ell_{\min}(\p K, \alpha)\quad\text{and} \quad \ell_{\min}^{\rm sym}(\p K,\rho):=\ell_{\min}^{\rm sym}(\p K, \alpha,\rho).
\end{equation}
They are independent of the choice of the Liouville form $\lambda$ because the (symmetric) systoles coincide with the minimal actions of (symmetric) closed characteristics on $(\p K, \rho)$.

\begin{remark}\label{rem: chacfoli}
More precisely, recall that the symplectic form $d\lambda_0=\sum_{j=1}^ndx_j\wedge
 dy_j$ induces the \emph{characteristic line bundle} $\mathcal{L}_{\p K}:=\ker (d\lambda_0|_{\p K})$ over $\p K$. This defines a 1-dimensional foliation of $\p K$ whose closed leaf $\gamma$ (i.e. an embedded circle whose tangent spaces lie in $\mathcal{L}_{\p K}$) is called a \emph{closed characteristic of $\p K$}. Its \emph{action} is defined by $\mathcal{A}(\gamma):=\int_{S^1} \gamma^*\lambda$ where $\lambda$ is a Liouville form on $K$ such that $d\lambda=d\lambda_0$. By Stokes's theorem $\mathcal{A}(\gamma)$ is independent of the choice of $\lambda$.
 A closed characteristic $\gamma$ on $\p K$ is called \emph{symmetric} if $\gamma$ is invariant under $\rho$.
Now if $\lambda$ is chosen as above so that $(\p K,\lambda|_{\p K},\rho)$ is a real contact manifold, then there is the correspondence between (symmetric) closed characteristics of $\p K$ and (symmetric) periodic orbits on $(\p K, \lda|_{\p K})$. Indeed, any closed characteristic of $\p K$ is parametrized by a periodic Reeb orbit, and vice versa. The action of a closed characteristic and the period of the corresponding periodic orbit coincide. Therefore, \eqref{eq: systoleforlda} is independent of the choice of~$\lambda$, but does depend on the hypersurface $\p K$ and the symplectic form $d\lambda_0$.
\end{remark}

%{\bf (i)} Now if $\lambda$ is chosen as above so that $(\p K,\lambda|_{\p K},\rho)$ is a real contact manifold, then the correspondence between (symmetric) closed characteristics and (symmetric) periodic orbits on $\p K$ preserves the action of  closed characteristics and the period of the corresponding periodic orbits. 
 %This shows that \eqref{eq: systoleforlda} does not depend on the choice of~$\lambda$.}

% \begin{equation}\label{eq: systoleforlda}
%\begin{aligned}
%\ell_{\min}(\p K) &:=\inf\{\mathcal{A}(\gamma)>0 \mid \text{$\gamma$ is a closed characteristic on $\p K$}\}, \\
%\ell_{\min}^{\rm sym}(\p K,\rho) &:=\inf\{\mathcal{A}(\gamma)>0 \mid \text{$\gamma$ is a symmetric closed characteristic on $\p K$}\}.
%\end{aligned} 	
%\end{equation}

The main result of this note is the following estimate on the symmetric ratio for symmetric convex hypersurfaces.
\begin{theorem} \label{mainthm} Let $K \subset \R^{2n}$ be a compact  and  convex domain with smooth boundary which is invariant under an anti-symplectic involution $\rho$ of $\R^{2n}$. Assume that $\Fix(\rho) \cap \p K \neq \emptyset$. Then the symmetric ratio of the symmetric convex hypersurface $(\p K, \rho)$ satisfies
\begin{equation}\label{main:estimate}
1 \leq \mathfrak{R}(\p K, \rho) \leq 2.
\end{equation}
In particular, on the boundary $\p K$, there exists a symmetric periodic orbit of period less than or equal to $ 2 \ell_{\min}(\p K)$. 
\end{theorem}

%\begin{remark}
%{\color{blue}The assumption that $\Fix(\rho) \cap \p K \neq \emptyset$ holds if $\rho$ is \emph{exact} i.e. $\rho^*\lda_0 = -\lda_0$. See Example \ref{ex: admLagconditions}.} 
%\end{remark}

It is particularly interesting to ask when the symmetric ratio is exactly equal to one. This means that the smallest period among all periodic orbits can be realized by symmetric one. 
In Section \ref{sec: Examples} we examine this question with explicit examples of symmetric hypersurfaces. 
Smooth starshaped toric domains, for instance, admit a family of anti-symplectic involutions including the complex conjugation, and we can explicitly understand its Reeb flows. We observe in Section \ref{sec: smoothtoricdomains} that the symmetric ratio in this case is always equal to one even without convexity. 
%On smooth starshaped toric domains, for instance, we can explicitly understand its Reeb flow, and we observe that the symmetric ratio is equal to one with respect to a family of anti-symplectic involutions including the complex conjugation. See Section \ref{sec: smoothtoricdomains}. 
On the other hand, there are symmetric starshaped domains whose boundary has the symmetric ratio bigger than 1. In Section \ref{sec: starshaped>1}, we construct such examples by perturbing the standard contact form on the unit sphere following the Bourgeois' perturbation scheme for Morse--Bott contact forms \cite[Section 2.2]{Bo02}.

\begin{remark}\label{rem: non-symmtricwithminimal} Even if the symmetric ratio is equal to one, there can exist a non-symmetric periodic orbit of the smallest period. Moreover, a symmetric periodic orbit of  the smallest period might not be unique. For example, consider the unit round sphere $S^{2n-1} \subset \R^{2n}\equiv \C^n$ for $n \geq 2$ with the complex conjugation $\rho_0$. The contact form is given by the restriction of the Liouville form as in Example \ref{exam:first}. The associated Reeb flow is periodic, and the periodic orbit $\gamma$ through $z \in S^{2n-1}$ can be parametrized as $\gamma(t) = e^{2 i t}z$, $ t \in \R$.  Then $\gamma$ is symmetric with respect to $\rho_0$ if and only if $\gamma(t_0)\in \R^n $ for some $t_0\in \R$. 
\end{remark}

Another interesting aspect of the estimate \eqref{main:estimate} is that it gives a uniform upper bound of the symmetric ratio for convex hypersurfaces in $\R^{2n}$. Such an upper bound does not necessarily exist for a larger class of hypersurfaces. For example, in Section \ref{sec: Bordeauxexample}, we exhibit symmetric starshaped hypersurfaces, which are Bordeaux-bottle-shaped, whose symmetric ratio is arbitrary large. In Section \ref{rmk:counterexample}, we  provide examples of restricted contact type, not starshaped, hypersurfaces in Hamiltonian systems whose symmetric ratio is also arbitrary large.
%It might be interesting to ask for which classes of symmetric hypersurfaces, containing convex hypersurfaces, the symmetric ratio admits a uniform upper bound. 

Our discussions up to now suggest the following questions:
\begin{itemize}
	\item Is the symmetric ratio for symmetric convex hypersurfaces in $\R^{2n}$ equal to exactly one?
	%\item {\color{blue} Can we find a uniform upper bound of the symmetric ratio for symmetric starshaped hypersurfaces in $\R^{2n}$?}
	\item Under what conditions on real contact manifolds  can we find a uniform upper bound of its symmetric ratio? For example, one can consider \emph{dynamical convexity} for contact manifolds as a substitute of geometric convexity.
\end{itemize}

\begin{remark}
A convex body $K \subset \R^{2n}$, i.e. a compact convex subset in $\R^{2n}$ with non-empty interior, is called \emph{centrally symmetric} if it is invariant under the antipodal map on $\R^{2n}$. Note that the antipodal map is not anti-symplectic
but symplectic. It is shown in Akopyan--Karasev \cite[Corollary 2.2]{AkKa17} that any closed characteristic of minimal action on the boundary of $K$ is itself centrally symmetric, cf. Remark \ref{rem: non-symmtricwithminimal}.
% It is worth noting that this result does not necessarily hold without convexity. One can construct a Bordeaux-bottle-shaped body, invariant under the antipodal map, having a closed characteristic on a thin bottle neck which is not centrally symmetric but is of minimal action, cf. Remark \ref{rem: Bordeaux-bottle}.
\end{remark}

%\begin{remark}
%For a convex body $K$ in $\R^{2n}$ which is \emph{centrally symmetric}, i.e. $K = -K$, it is shown in Akopyan--Karasev \cite[Corollary 2.2]{AkKa17} that any closed characteristic of minimal action on the boundary is itself centrally symmetric, cf. Remark \ref{rem: non-symmtricwithminimal}. 
%\end{remark}

 In Section \ref{sec: Floerhomologycapacities}, we present an approach to obtain the upper bound in Theorem \ref{mainthm} employing symplectic capacities from Floer theory. We first bound the symmetric ratio from above in terms of the symplectic homology capacity  (the $\SH$ capacity) $c_{\SH}$ and the wrapped Floer homology capacity  (the $\HW$ capacity) $c_{\HW}$. 
 An essential ingredient is the recent result of Abbondandolo--Kang \cite{AbonKang} and Irie \cite{Irie} showing for \emph{convex} domains that the systole $\ell_{\min} (\p K )$ coincides with the $\SH$ capacity $c_{\SH}(K)$. Together with the spectral property of the $\HW$ capacity in Proposition \ref{prop: propertiesofHWcapa}, we deduce that
 $$
\frac{\ell_{\min}^{\rm sym}(\p K, \rho) }{\ell_{\min} (\p K )} \leq  \frac{2 c_{\HW}(K, \rho)}{c_{\SH}(K)}.
 $$
\noindent
We can then bound the ratio of the capacities from above using Floer theory. In Section \ref{sec: closedopenmaps} we recall a construction of well-known comparison homomorphisms in Floer homology, called \emph{closed-open maps}. They are defined by counting certain Floer disks with one interior puncture (asymptotic to a Hamiltonian 1-orbit) and one boundary puncture (asymptotic to a Hamiltonian 1-chord) with Lagrangian boundary condition. See Figure \ref{fig: chim}. We call them \emph{Floer chimneys} as in \cite[Figure 11]{Al08}. Closed-open maps are compatible with the action filtrations on the Floer homologies in the sense of Theorem \ref{thm: closedopenmap}.  As also observed in \cite{BormanMclean}, it is rather straightforward to obtain the desired upper bound from the existence of filtered closed-open maps.

At the heuristic level the underlying geometric idea is the following. By the spectral properties, the $\SH$ capacity $c_{\SH}(K)$ is the action $\mathcal{A}(\gamma)$ of a periodic orbit $\gamma$ on $\p K$ and the $\HW$ capacity $c_{\HW}(K, \rho)$ is the action $\mathcal{A}(x)$ of a chord  $x$ on $(\p K, \p \Fix(\rho))$. Closed-open maps in principle tell  us that there exists a $J$-holomorphic chimney %(i.e.\ a $J$-holomorphic disk with one interior puncture, one boundary puncture, and Lagrangian boundary conditions) 
asymptotic to $\gamma$ at the interior puncture and asymptotic to $x$ at the boundary puncture. Since the energy of $J$-holomorphic chimneys is necessarily non-negative, one has $c_{\HW}(K, \rho)=\mathcal{A}(x) \leq \mathcal{A}(\gamma) = c_{\SH}(K)$ by  Stokes'~theorem.

The wrapped Floer homology capacity for symmetric domains can be seen as a symplectic capacity for symplectic manifolds with symmetries, which we call a \emph{real symplectic capacity}. The upper bound in Theorem \ref{mainthm} hinges on relationships between real and non-real symplectic capacities. In Section \ref{sec:examples} we discuss further examples of real symplectic capacities which might be of independent interest. See also \cite{CHLS07} for more information on symplectic capacities.

\begin{remark} One finds a motivation to study the symmetric systole in the context of the planar circular restricted three-body problem (PCR3BP). This problem studies the motion of a massless body influenced by two bodies of positive mass according to Newton's law of gravitation, where the two massive bodies move in circles about their common center of mass, and the massless body is confined to the plane determined by the two bodies. Denote by $c_*$ the energy value of the Hamiltonian $H$ of the PCR3BP such that for every $c<c_*$, the  level set $H^{-1}(c)$ contains two bounded components near either massive body. In what follows, we concentrate on one of the two bounded components, denoted by~$\Sigma_c$. It is invariant under the   anti-symplectic involution $\rho$ whose fixed point set projects into the configuration space $\R^2$ as a subset of the horizontal axis. In \cite{Birkhoff} Birkhoff found a Reeb chord on $\Sigma_c$  via shooting argument and closed it up using $\rho$ to obtain   a symmetric periodic orbit, called a \emph{retrograde periodic orbit}. In a real-world situation, a direct periodic orbit is more important since most orbits of moons in the solar system are direct. However, Birkhoff did not give an analytic proof of the existence of a direct periodic orbit.   Instead, he conjectured that for each $c<c_*$, the retrograde periodic orbit on $\Sigma_c$ bounds a disk-like global surface of section. Birkhoff believed that a  fixed point of the associated first return map, whose existence is assured by Brouwer's translation theorem, corresponds to a direct periodic orbit. One way to prove this conjecture is to look at the period of the retrograde periodic orbit.   Indeed,  the SFT-compactness theorem says  that if the retrograde periodic orbit has  the smallest period,  then   this would imply  Birkhoff's conjecture. For details, we refer to a beautiful exposition~\cite{book}.
\end{remark}

\subsection*{Acknowledgement} The authors cordially thank Urs Frauenfelder, Jungsoo Kang, and Felix Schlenk for fruitful discussions, Yaron Ostrover for helpful comments. They also thank the anonymous referee for valuable suggestions. A part of this work was done while MK visited Korea Institute for Advanced Study. The authors are grateful for its warm hospitality.
JK is supported by a KIAS Individual Grant MG068002 at Korea Institute for Advanced Study. SK is supported by the grant 200021-181980/1 of the Swiss National Foundation. MK is supported by the SFB/TRR 191 \emph{Symplectic Structures in Geometry, Algebra and Dynamics}, funded by the DFG.

\section{Examples}\label{sec: Examples} In this section we discuss examples for the symmetric ratio \eqref{eq: defofratio} on various symmetric hypersurfaces.

\subsection{In dimension two} Let $W$ be a subset of $\R^{2}$ that is diffeomorphic to a closed disc and  invariant under an   anti-symplectic involution $\rho$. There exists a unique simply covered periodic orbit $\gamma$, which is a parametrization of the $\rho$-invariant circle $\p W$. Moreover,  $\gamma$ is $\rho$-symmetric. It follows that $\ell_{\min}(\p W) = \ell_{\min}^{\mathrm{sym}}(\p W)$ and hence $\mathfrak{R}(\p W, \rho ) = 1$.

\subsection{Ellipsoids}  Given  $a_j \in \R_{> 0}$, $j=1,\ldots, n$, the associated ellipsoid is given by
  \[
   E(a_1, \dots, a_n) : = \Bigg\{z \in \C^{n} \;\bigg|\; \sum_{j=1}^n \frac{\pi|z_j|^2}{a_j} \leq 1 \Bigg\}  .
\]
With respect to the standard contact form, i.e.\ the restriction of the Liouville form from Example \ref{exam:first},  the complex conjugation $\rho_0$ provides  an anti-contact involution on the boundary $\p E(a_1, \ldots, a_n)$. The Reeb flow can explicitly be written by coordinate-wise rotations on $\C^n$. Periodic orbits are of the form
\begin{equation}\label{eq: perorbellip}
\gamma (t)= (e^{\frac{2\pi it}{a_1}}z_1, e^{\frac{2\pi it}{a_2}}z_2, \dots, e^{\frac{2\pi it}{a_n}}z_n)
\end{equation}
    for some $(z_1, z_2, \dots, z_n) \in \p E(a_1, \dots, a_n)$. Assuming $a_1 \leq a_2 \leq \cdots \leq a_n$ without loss of generality, the periodic orbit $\gamma_1 (t) = (e^{\frac{2\pi it}{a_1}}z_1, 0, \dots, 0)$ with $z_1 \neq 0$ attains the minimal period and is symmetric with respect to $\rho_0$. Hence the symmetric ratio is equal to one. In general, a periodic orbit of the form \eqref{eq: perorbellip} is $\rho_0$-symmetric if and only if $\gamma(t_0)\in \R^n$ for some $t_0\in \R$.

\subsection{Smooth starshaped toric domains} \label{sec: smoothtoricdomains}
Define the  {moment map}  $\mu \colon \C^n \to \R^n_{\geq 0}$ as 
$$
\mu( z_1, \ldots, z_n) = \pi ( \lvert z_1 \rvert^2, \ldots, \lvert z_n \rvert ^2 ).
$$
It is invariant under the exact anti-symplectic involution
\begin{equation}\label{eq:antisympinvol}
\rho_{\theta}(z) = (e^{i\theta_1}\overline z_1, \dots, e^{i\theta_n}\overline z_n)
\end{equation}
for each  $ \theta = (\theta_1, \ldots ,\theta_n) \in \R^n$. For a domain $\Omega \subset \R^n_{\geq 0}$, the preimage $X_{\Omega} := \mu^{-1}( \Omega ) \subset \C^n$ is called a \textit{toric domain}. Note that any toric domain is $\rho_{\theta}$-invariant. For example, the ellipsoid $E(a_1, \ldots, a_n)$ is a smooth toric domain associated to the simplex
\[
\Omega = \Bigg\{  x \in \R^n_{ \geq 0} \;\bigg|\;   \sum_{j=1}^n \frac{x_j}{a_j} \leq 1   \Bigg\}.
\]
A toric domain is not necessarily smooth, but, in this note, we only consider \emph{smooth}  ones.

In what follows we assume that a domain $\Omega \subset \R^n _{\geq 0}$ is smooth, compact,  and  starshaped (with respect to the origin). Then the associated toric domain $X_\Omega\subset \C^n$ is a smooth toric domain that is compact and starshaped.  
 %We shall show that   $\mathfrak{R}(\p X_\Omega,\rho_\theta)=1$ for every $\theta \in \R^n$. Fix $\theta \in \R^n$ and note  that   $X_{\Omega}$ is invariant under the $\mathbb{T}^n$-family of the exact symplectic involutions
%\[
% \sigma_{\phi}(z) = ( e^{ i \phi_1} z_1,   \ldots,  e^{ i \phi_n}  z_n), \quad \phi = (\phi_1, \ldots, \phi_n) \in \R^n.
% \]
%For each $z \in \C^n$, we have  $\{ \mu (\sigma_{\theta}(z)) \mid \theta \in \R^n \} = \{ \pi ( \lvert z_1 \rvert^2 , \ldots, \lvert z_n \rvert^2 ) \} \in \R^n_{\geq 0}$. 
%Since $\sigma_{\phi}$ is symplectic, if $\gamma$ is a periodic orbit on $\p X_{\Omega}$, then so is $\sigma_{\phi}(\gamma)$.  
 We shall show that   $\mathfrak{R}(\p X_\Omega,\rho_\theta)=1$ for every $\theta \in \R^n$. 
 
Note  that   $X_{\Omega}$ is invariant under the $\mathbb{T}^n$-family of the exact symplectomorphisms
\[
 \sigma_{\phi}(z) = ( e^{ i \phi_1} z_1,   \ldots,  e^{ i \phi_n}  z_n), \quad \phi = (\phi_1, \ldots, \phi_n) \in \R^n.
 \]
%For each $z \in \C^n$, we have  $\{ \mu (\sigma_{\theta}(z)) \mid \theta \in \R^n \} = \{ \pi ( \lvert z_1 \rvert^2 , \ldots, \lvert z_n \rvert^2 ) \} \in \R^n_{\geq 0}$. 
If $\gamma$ is a periodic orbit on $\p X_{\Omega}$, then so is $\sigma_{\phi}(\gamma)$. In fact, each fiber torus $\mu^{-1}(w)$, $w\in \Omega$, is foliated by periodic orbits, see e.g. \cite[Section 2.2]{GH}, and hence any periodic orbit on $\p X_\Omega$ is contained in a fiber torus.

For a fixed $\theta \in \R^n$, each fiber torus contains a $\rho_{\theta}$-symmetric periodic orbit. Indeed, in view of the fact that $\rho_{\theta}^* R = - R$, where $R$ is the Reeb vector field on $\p X_{\Omega}$, a periodic orbit $\gamma$ is $\rho_{\theta}$-symmetric if and only if $\gamma(\R) \cap \mathrm{Fix}(\rho_{\theta}) \neq \emptyset$. For a periodic orbit $\gamma$ in a fiber torus $\mathcal{T}$, it is always possible to find $\phi\in \R^n$ such that $\sigma_\phi(\gamma)$ intersects $\Fix(\rho_\theta)$. Then $\sigma_{\phi}(\gamma)$ is a $\rho_{\theta}$-symmetric periodic orbit in $\mathcal{T}$. 
%We find  $\phi \in \R^n$ and rotate $\gamma$ using $\sigma_{\phi}$ so that  $\sigma_{\phi}(\gamma)$ intersects $\mathrm{Fix}(\rho_{\theta})$. 

As all periodic orbits belonging to the same fiber torus have the same period, this implies that   $\mathfrak{R}(\p X_\Omega,\rho_\theta)=1$. Actually, for every periodic orbit $\gamma$ on $\p X_{\Omega}$, there exists $\theta = \theta(\gamma) \in \R^n$ such that $\gamma$ is a $\rho_{\theta}$-periodic orbit. 

Recall that a toric domain $X_{\Omega}$ is said to be \emph{convex} if 
\[
\widehat{\Omega}= \{ (x_1, \ldots, x_n ) \in \R^n \mid ( \lvert x_1 \rvert, \ldots, \lvert x_n \rvert) \in \Omega \}  \subset \R^n
\]
 is convex. In this case, we know which (symmetric) periodic orbit attains the smallest period. From the convexity of $\widehat{\Omega}$, we can show by computing the Reeb vector field that the fiber orbit at a point of $\p \Omega$ along a coordinate axis, i.e.\ an intersection point of $\p X_\Omega$ with a coordinate axis, attains the smallest period. Moreover, it is also obvious from the Reeb flow that such a periodic orbit is   $\rho_{\theta}$-symmetric for every $\theta \in \R^n$.

\subsection{Starshaped domains with the symmetric ratio slightly bigger than one}\label{sec: starshaped>1}

For every $\theta\in \R^n$ we can construct a $\rho_\theta$-symmetric starshaped domain $K$ in $\R^{2n}$ with $\mathfrak{R}(\p K,\rho_\theta)>1$, where $\rho_{\theta}$ is defined as in \eqref{eq:antisympinvol}, by perturbing the round sphere. Without loss of generality, we only consider the case of the complex conjugation $\rho=\rho_0$.

%For every $\theta\in \R^n$ there exists a $\rho_\theta$-symmetric starshaped domain $K$ in $\R^{2n}$ with $\mathfrak{R}(\p K,\rho_\theta)>1$, where $\rho_{\theta}$ is defined as in \eqref{eq:antisympinvol}.  Without loss of generality, we only consider the case of the complex conjugation $\rho=\rho_0$.

Let $B\subset (\R^{2n},\lambda_0)$ denote the closed unit ball. For $h\in C^\infty(\p B, \R)$ with $h\ge 1$, we define the starshaped domain in $\R^{2n}$
$$
K_h:=B\cup_{\p B} \{(r,x)\in [1,\infty)\times \p B \mid x\in \p B,\ r\le h(x)\}
$$
by attaching the graph of $h$ along the boundary $\p B$ via the Liouville flow of $\lambda_0$.  Note that $\p K_h$ is contactomorphic to the unit sphere $\p B$ equipped with the contact form $h\alpha_0$. Since the Reeb flow $\phi^t$ on $(\p B,\alpha_0)$ satisfies $\rho\circ \phi^{-t}\circ \rho=\phi^t$, the involution $\rho$ of $\p B$ descends to the involution $\bar{\rho}$ of $\p B/S^1\cong \C P^{n-1}$, where the $S^1$-action on $\p B$ is given by the Reeb flow. 

Take a $\bar{\rho}$-invariant Morse function $\bar{f}\ge 0$ on $\p B/S^1$ which attains the minimum precisely at a pair of two critical points away from the fixed point set of $\bar{\rho}$. We write $f\in C^\infty(\p B,\R)$  for the lifting of $\bar{f}$. Set $h_\epsilon:=1+\epsilon f$ for $\epsilon>0$. Since $h_\epsilon$ is $\rho$-invariant, the starshaped domain $K_{h_\epsilon}$ is symmetric. We claim that for $\epsilon>$ sufficiently small we have $\mathfrak{R}(\p K_{h_\epsilon},\rho)>1$, but this will be close to 1. We denote by $T_{\min}$ the minimal period of the Reeb flow of the standard contact sphere $(\p B, \alpha_0)$. Recall from Bourgeois \cite[Section 2.2]{Bo02} that for any $T>T_{\min}$ there exists $\epsilon>0$ such that the periodic orbits of $(\p B,\alpha_\epsilon:=h_\epsilon \alpha)$ of period less than $T$ are non-degenerate and correspond to the critical points of $\bar{f}$. For a critical point $\bar{x}$ of $\bar{f}$ the corresponding periodic orbit is the $S^1$-fiber $\gamma_{\bar{x}}$ of the fibration $\p B\to \p B/S^1$ at $\bar{x}$, and its period~is given by $T_{\min}h_\epsilon(x)$ for any lift $x\in \p B$ of $\bar{x}$.  The periodic orbit $\gamma_{\bar{x}}$ is symmetric if and only if $\bar{x}$ is a fixed point of $\bar{\rho}$. Now we take $T>0$ slightly bigger than $T_{\min}$. For $\epsilon>0$ small enough, the minimum period of non-symmetric periodic orbits is strictly smaller than the minimum period of symmetric periodic orbits. This shows~that $\mathfrak{R}(\p K_{h_\epsilon},\rho)>1$. It is worth noting that    $\mathfrak{R}(\p K_{h_\epsilon},\rho)$ can be arbitrarily close to one.

\subsection{Bordeaux-bottle-shaped hypersurfaces of arbitrarily large symmetric ratio} \label{sec: Bordeauxexample}
Recall that the classical Bordeaux-bottle $K\subset \R^{2n}$ is a smooth starshaped domain obtained by gluing a thin neck, modeled on the symplectic 2-subspace $\R^2\times \{0\}$, along the boundary of the unit ball $B$, see \cite[Section~3.5]{HZbook}.
Let $\rho$ denote the complex conjugation on $\R^{2n}$.
For a given symplectic 2-subspace $V\subset \R^{2n}$ with $\rho(V)\cap V=\{0\}$, we glue the two thin necks, associated to $V$ and $\rho(V)$ respectively, along the boundary $\p B$ to form a $\rho$-symmetric Bordeaux-bottle-shaped domain $K_V$ with two necks. See Figure \ref{Fig:counter111}. Since $\rho$ sends one neck to the other, any periodic orbits on the thin necks are not symmetric.
Moreover, symmetric periodic orbits on $\p K_V$ exist on $\p B$, and the period of any symmetric periodic orbits of $\p K_V$ is uniformly bounded from below. Making the necks narrow, we obtain a symmetric starshaped domain $K_V$ with arbitrarily large symmetric ratio. 
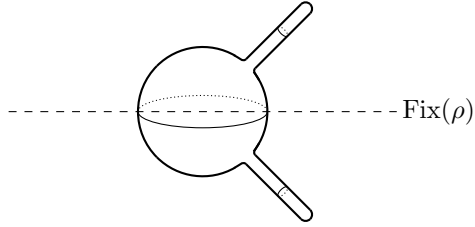
\begin{figure}[h]
\begin{center}
\begin{tikzpicture}[scale=0.43]
  \begin{scope}[rotate=-45]
 \draw[thick] (-0.2,2.1) to (-0.2, 4.5);
  \draw [thick] ( -2 ,0) arc ( 180:105:2cm and  2cm);
  \draw [thick] ( -2 ,0) arc ( 180:225:2cm and  2cm);
    %\draw [thick] ( 0.2,2.1) arc ( 75:45:2cm and  2cm);

  \draw[thick] (-0.2, 2.1) to [out=270, in=20](-0.55, 1.92);
\draw[thick] (-0.2, 4.5) to [out=90, in=180] (0, 4.7);
\begin{scope}[xscale=-1]
 \draw[thick] (-0.2,2.1) to (-0.2, 4.5);
 \draw[thick] (-0.2, 2.1) to [out=270, in=20](-0.55, 1.92);
\draw[thick] (-0.2, 4.5) to [out=90, in=180] (0, 4.7);
\end{scope}
 \draw   (-0.2, 3.5)  arc ( 180:360:0.2cm and  0.08cm);
   \draw[densely dotted]   (-0.2, 3.5)  arc ( 180:0:0.2cm and  0.08cm);
 
 \end{scope}
\begin{scope}[yscale=-1]
 \begin{scope}[rotate=-45]
 \draw[thick] (-0.2,2.1) to (-0.2, 4.5);
  \draw [thick] ( -2 ,0) arc ( 180:105:2cm and  2cm);
  \draw [thick] ( -2 ,0) arc ( 180:225:2cm and  2cm);
    %\draw [thick] ( 0.2,2.1) arc ( 75:45:2cm and  2cm);

  \draw[thick] (-0.2, 2.1) to [out=270, in=20](-0.55, 1.92);
\draw[thick] (-0.2, 4.5) to [out=90, in=180] (0, 4.7);
\begin{scope}[xscale=-1]
 \draw[thick] (-0.2,2.1) to (-0.2, 4.5);
 \draw[thick] (-0.2, 2.1) to [out=270, in=20](-0.55, 1.92);
\draw[thick] (-0.2, 4.5) to [out=90, in=180] (0, 4.7);
\end{scope}
 \draw   (-0.2, 3.5)  arc ( 180:360:0.2cm and  0.08cm);
   \draw[densely dotted]   (-0.2, 3.5)  arc ( 180:0:0.2cm and  0.08cm);
 
 \end{scope}
\end{scope}
 
  \draw [thick] ( 2 ,0) arc ( 0:38:2cm and  2cm);
\begin{scope}[yscale=-1]
  \draw [thick] ( 2 ,0) arc ( 0:38:2cm and  2cm);
\end{scope}

\draw[dashed] (-6,0) to (6,0);
\node at (7.3,0) {$\mathrm{Fix}(\rho)$};
  \draw   (-2,0)  arc ( 180:360:2cm and  0.5cm);
  \begin{scope}[yscale=-1]
 \draw[densely dotted]   (-2,0)  arc ( 180:360:2cm and  0.5cm);
\end{scope}
    \end{tikzpicture}
\end{center}
\caption{A symmetric Bordeaux-bottle-shaped domain having two necks}
\label{Fig:counter111}
\end{figure}
Below, we provide a detailed account of this construction.

Consider the symplectic 2-subspace $V$ in $(\R^{2n},\ow_0=\sum_{j=1}^ndx_j\wedge dy_j)$ spanned~by
\begin{align*}
	v_1 &= (1,0,0,1,0,\dots,0), \\
	Jv_1 &= (0,1,-1,0,0,\dots,0),
\end{align*}
where $J$ denotes the standard complex structure on $\C^n\cong \R^{2n}$.
A simple computation shows that
\begin{equation}\label{eq: V_disj}
\rho(V)\cap V=\{0\}.	
\end{equation}
Since $\rho$ is anti-symplectic, $\rho(V)$ is a symplectic 2-subspace. We then obtain a symplectic orthogonal decomposition
$$
\R^{2n}=V\oplus \rho(V)\oplus W,
$$
where $W=(V\oplus \rho(V))^\perp$ denotes the symplectic complement of $V\oplus \rho(V)$. Since $V\oplus \rho(V)$ is $\rho$-invariant, so is $W$.
Using the Gram--Schmidt process \cite[Lemma~2.6.6]{MS17book}, we can construct a unitary basis on $\R^{2n}$,
\begin{equation}\label{eq: new_coord}
\{v_1,Jv_1,\dots,v_n,Jv_n\},	
\end{equation}
such that 
\begin{itemize}
	\item $v_2:=\rho(v_1)$ and $Jv_2=J\rho(v_1)=-\rho(Jv_1)$,
	\item $\{v_3,Jv_3,\dots, v_n,Jv_n\}$ is a unitary basis for $W$.
\end{itemize}
We denote by
\begin{equation}\label{eq: coord}
({\bf x}',{\bf y}')=(x_1',y_1',\dots, x_n',y_n')	
\end{equation}
the symplectic coordinates on $\R^{2n}$ with respect to the basis \eqref{eq: new_coord}. Then $K_V$ is defined to be a $\rho$-symmetric smooth starshaped domain in $\R^{2n}$ consisting~of
\begin{itemize}
	\item a bounded piece of the neck $N_V=\{({\bf x}',{\bf y}')\mid (x_1')^2+(y_1')^2 \le \epsilon \}$ of $V$;
	\item a bounded piece of the neck $N_{\rho(V)}=\{({\bf x}',{\bf y}')\mid (x_2')^2+(y_2')^2 \le \epsilon \}$ of $\rho(V)$;
	\item the unit ball $B=\{({\bf x}',{\bf y}') \mid \sum_{j=1}^n (x_j')^2+(y_j')^2 \le 1\}$.
\end{itemize}
Here, a smoothing procedure is required as in the well-known case of a Bordeaux-bottle having one neck. 
We emphasize that the smoothing procedure in the standard Bordeaux-bottle is still enough for our case, since the gluing regions of $\p N_V$ and $\p N_{\rho(V)}$ along $\p B$ are disjoint due to \eqref{eq: V_disj}. 
Moreover, the unit ball in the coordinates \eqref{eq: coord} coincides with the unit ball in the standard coordinates.

We claim that the symmetric ratio $\mathfrak{R}(\p K_V,\rho)$ can be arbitrarily large by choosing $\epsilon>0$ small enough.
Since $V$ is chosen to be symplectic, every periodic orbit on the boundary of the neck $N_V$ is of the form
\begin{equation*}\label{eq: orbit}
\gamma(t)=(w_1e^{\frac{2it}{\epsilon}},w_2,\dots,w_n),	
\end{equation*}
where the identifications \eqref{eq: coord} and $w_j=x'_j+iy'_j$ are used. They have small periods depending on $\epsilon>0$. The similar holds for periodic orbits on $\p N_{\rho(V)}$. Thanks to~\eqref{eq: V_disj}, any periodic orbits on the necks $\p N_V$ and $\p N_{\rho(V)}$ are not symmetric under~$\rho$. As mentioned before, symmetric periodic orbits of $\p K_V$ exist on the boundary $\p B$, and the period of any symmetric periodic orbits of $\p K_V$ is uniformly bounded from below. Therefore, the claim follows.

\subsection{Hypersurfaces of arbitrarily large symmetric ratio in Hamiltonian systems}
\label{rmk:counterexample} 
Recall that a hypersurface $\Sigma$ in $\R^{2n}$ is called of \emph{restricted contact type} if there exists a Liouville vector field $X$ which is  defined in a neighborhood of the hypersurface and which is transverse to $\Sigma$. %By definition, a Liouville vector field $X$ satisfies $\mathcal{L}_X \omega_0= \omega_0$. 
If $\Sigma$ is of restricted contact type with the radial vector field $X = \tfrac{r}{2} \p_r$, then it is starshaped.  Here we provide  a restricted contact type, but not starshaped, hypersurface of arbitrarily large symmetric ratio. 

%Therefore, it is reasonable to ask the above question for a smaller class of symmetric hypersurfaces, for example, starshaped hypersurfaces, than the class of restricted contact type ones.   We do not know an  example of starshaped and non-convex invariant hypersurface  for which the symmetric ratio is bigger than one (cf. Remark \ref{rem:bottle}). 

Consider a mechanical Hamiltonian $H(q,p) = \tfrac{1}{2} \lvert p \rvert^2 + V(q)$, $(q,p) \in \R^2 \times \R^2$, where the potential $V$ is invariant under the involution $(q_1, q_2) \mapsto (- q_1, q_2)$. It follows that $H$ is invariant under the   anti-symplectic involution $\rho(q_1,q_2,p_1,p_2) = (-q_1,q_2,p_1,-p_2)$, and hence for every $E \in \R$, the energy level set $H^{-1}(E)$ is $\rho$-invariant.  We assume the following.
\begin{itemize}
\item There exist exactly two saddle points $(\pm a,0)$ of $V$ such that $V( \pm a ,0) =0$.
\item For $E>0$ small enough,   $H^{-1}(-E)$ consists of three 3-spheres. 
\end{itemize}
The first condition implies that the equilibriums $(\pm a, 0,0,0)$  of $H$ are of saddle-center type, and the second condition implies that the energy level $H^{-1}(E)$ for $E=0$ and   for $E$ small enough project into the position space $\R^2$ as in Figure \ref{Fig:counter}. 
    \begin{figure}[h]
\begin{center}
\begin{tikzpicture}[scale=0.6, >=latex]
	\draw [thick, fill=gray!20] plot [smooth cycle, tension=0.7] coordinates {
(0, 1.5)
(0.5, 1.2)
(1, 0)
(2, -1)
(3,0)
(2.4, 1)
(1, 0)
(0.5, -0.8)
(0, -1)
(-0.5, -0.8)
(-1, 0)
(-2.4, 1)
(-3, 0)
(-2, -1)
(-1, 0)
(-0.5, 1.2)
};
		\draw [->] (-3.5,0) to (3.5,0);
	\draw [->] (0,-1.7) to (0,2.1);
	
   \draw [fill] (1,0) circle [radius=0.07];
  \draw [fill] (-1,0) circle [radius=0.07];	
	
\node at (1.2, -1.4) {$E=0$};
\node at (3.3,-0.5) {$q_1$};
\node at (0.5, 1.8) {$q_2$};	
	
\begin{scope}[xshift=9cm]
	\draw [thick, fill=gray!20] plot [smooth cycle, tension=0.7] coordinates {
(0, 1.5)
(0.5, 1.2)
(0.85, 0.4)
(1.3, 0.2)
(2.4, 1)
(3, 0)
(2, -1)
(1.3, -0.2)
(0.85, -0.3)
(0.5, -0.8)
(0, -1)
(-0.5, -0.8)
(-0.85, -0.3)
(-1.3, -0.2)
(-2, -1)
(-3, 0)
(-2.4, 1)
(-1.3, 0.2)
(-0.85, 0.4)
(-0.5, 1.2)
};
		\draw [->] (-3.5,0) to (3.5,0);
	\draw [->] (0,-1.7) to (0,2.1);
\node at (1.7, -1.4) {$E>0$ small};
\node at (3.3,-0.5) {$q_1$};
\node at (0.5, 1.8) {$q_2$};

\draw  [red, thick] (1.1, 0) arc (0:360:0.05cm and  0.1cm);
\draw  [red, thick] (-1, 0) arc (0:360:0.05cm and  0.1cm);

\end{scope}

\end{tikzpicture}
    \caption{The projections of the energy levels $H^{-1}(E)$ into the position space $\R^2$}
    \label{Fig:counter}

\end{center}
%\caption{An invariant starshaped hypersurface of arbitrarily large symmetric ratio}
%\label{Fig:counter}
\end{figure}
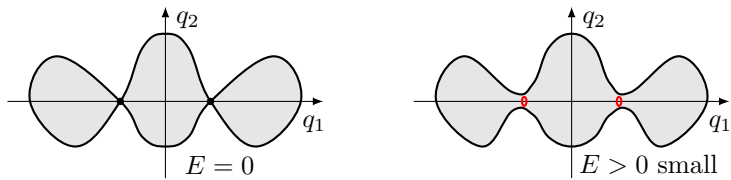
For every $E>0$ sufficiently small, $H^{-1}(E)$ is not starshaped, but of restricted contact type as  $H$ is of mechanical type.  Since $(\pm a, 0,0,0)$ are of saddle-center type, in view of a well-known theorem by Lyapunov, if $E>0$ is small enough, $H^{-1}(E)$ carries periodic orbits $\gamma_1, \gamma_2= \rho(\gamma_1)$, called the  Lyapunov orbits (red curves in Figure \ref{Fig:counter}). As $E \to 0^+$, they converge to equilibriums.  Moreover, in a sufficiently small neighborhood of equilibriums, there exists no periodic orbit other than the associated Lyapunov orbit, and   periodic orbits that pass  this neighborhood have sufficiently large periods.   This in particular implies that if $E>0$ is small enough, then the periods of the Lyapunov orbits are extremely small, but the periods of other periodic orbits are bounded from below some positive constant.   As the Lyapunov orbits are not $\rho$-symmetric, we conclude that the symmetric ratio can be chosen arbitrarily~large.

\section{Closed-open maps}\label{sec: section3}

\subsection{Symplectic homology}\label{sec: SH} We briefly recall the construction of symplectic homology without technical details.  We refer the reader to \cite[Section 2]{BO_MB} for a detailed description. We work with Liouville domains, and prominent examples are starshaped domains in $\R^{2n}$ including smooth convex bodies. In this paper, we always use $\Z_2$-coefficients.

Let $(W, \lda)$ be a Liouville domain with a Liouville form $\lda$. This means $W$ is a compact smooth manifold with boundary and $\lda$ is a 1-form on $W$ such that $d\lda$ is symplectic and its Liouville vector field is positively transverse along the boundary. The restriction $\alpha : = \lda|_{\p W}$ of the Liouville form defines a contact form on the boundary, and we denote the contact boundary by $(\Sigma, \alpha) : = (\p W, \lda|_{\p W})$. The completion $(\widehat W, \widehat \lda)$ of the Liouville domain $(W, \lda)$ is an open symplectic manifold defined by attaching (a positive part of) the symplectization $([1, \infty) \times \Sigma, r\alpha)$ to the domain $(W, \lda)$ along the boundary via the Liouville flow. Here $r \in [1, \infty)$ denotes the Liouville coordinate. 

\begin{example}\label{ex: ballLiouv}
The closed unit ball $B^{2n} \subset \R^{2n}$ with the standard symplectic~form $\ow_0 = \sum_{j=1}^{n}dx_j \wedge dy_j$ is a Liouville domain with a Liouville form $\lda_0 = \tfrac{1}{2} \sum_{j=1}^n ( x_j dy_j - y_j dx_j)$. The contact type boundary is the standard contact sphere $(S^{2n-1}, \alpha_{0})$ with $\alpha_0 = \lda_0|_{\p B^{2n}}$. The completion of $B^{2n}$ recovers $\R^{2n}$. More generally, any starshaped domains  in $\R^{2n}$, including smooth convex ones, fit into our setup for Floer theory.
\end{example}

\subsubsection{Admissible Hamiltonians}

We take an \emph{admissible} time-dependent Hamiltonian $H_{S^1}: S^1 \times \widehat W \rightarrow \R$, meaning that all 1-periodic orbits of the Hamiltonian vector field $X_{H_{S^1}}$ is non-degenerate, $H_{S^1}$ is negative and $C^2$-small (and Morse) in the interior of $W \subset \widehat W$, and $H_{S^1}$ is linear at the end with respect to the Liouville coordinate $r$, independent of the time parameter $ t\in S^1$.
The derivative $H_{S^1}'(r)$ at the end is called the \emph{slope} of the Hamiltonian $H_{S^1}$. We assume that the slope is positive and not equal to the period of a periodic Reeb orbit in the contact boundary $(\Sigma, \alpha)$. See Remark \ref{rem: noperiodcontactregion}.
 \begin{remark}
Our convention for Hamiltonian vector fields is that $\ow(X_{H}, \cdot) = dH$. 
\end{remark}

Denote the set of contractible 1-periodic orbits of $H_{S^1}$ by $\mathcal{P}(H_{S^1})$. To each 1-periodic orbit $\gamma \in \mathcal{P}(H_{S^1})$ we can associate an integer called the Conley--Zehnder index $\CZ(\gamma)$ by taking a capping disk of $\gamma$. We assume that $c_1(TW)$ vanishes on $\pi_2(W)$ for well-definedness of the index $\CZ(\gamma)$. See \cite{BO_MB} for details on the index.

\subsubsection{Chain complex}\label{sec: symplecticchaincomplex} Let $J_{S^1} = \{J_t\}_{t \in S^1}$ be a time-dependent family of compatible almost complex structures on $(\widehat W, \widehat \lda)$ which is admissible in the sense of \cite{BO_MB}. The \emph{Floer chain group} $\CF_*(H_{S^1}, J_{S^1})$ for the pair $(H_{S^1}, J_{S^1})$ is a $\Z$-graded vector space over $\Z_2$, generated by the 1-periodic orbits of $\mathcal{P}(H_{S^1})$ and graded by the negative Conley--Zehnder index $|\gamma|=-\CZ(\gamma)$:
$$
\CF_k(H_{S^1}, J_{S^1}) = \bigoplus_{\substack{\gamma \in \mathcal{P}(H_{S^1}) \\ |\gamma| = k}} \Z_2 \langle \gamma \rangle.
$$
For two distinct 1-periodic orbits $\gamma_{\pm} \in \mathcal{P}(H_{S^1})$, define the moduli space of \emph{Floer cylinders} $\mathcal{M}(\gamma_-, \gamma_+, H_{S^1}, J_{S^1})$ from $\gamma_-$ to $\gamma_+$, modulo the natural $\R$-action, by 
\begin{align}\label{eq: SHmoduli}
\begin{split}
\mathcal{M}(\gamma_-, \gamma_+, H_{S^1}, J_{S^1}) = \{u: \R \times S^1 \rightarrow \widehat W \;|\; &\lim_{s \rightarrow \pm \infty} u(s, t) = \gamma_{\pm}(t), \\ &(du - X_{H_{S^1}}\otimes dt )^{0,1} = 0 \} / \R.
\end{split}
\end{align}
See the left in Figure \ref{fig: cylandstr}.

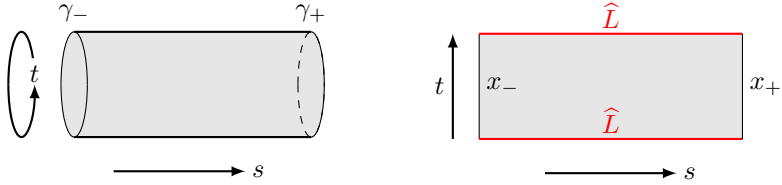
\begin{figure}[h]
\begin{subfigure}{0.45\textwidth}
   \centering
\begin{tikzpicture}[scale=0.35]

   \begin{scope}[>=latex]
%gray filling
\draw [gray!20, fill=gray!20] (-11.55, 2)--(-2.45, 2)--(-2.45, -2)--(-11.55, -2);
\draw  [gray!20, fill=gray!20] (-11.5, 2) arc (90:270:0.5cm and  2cm);
\draw  [gray!20, fill=gray!20] (-2.5, -2) arc (-90:90:0.5cm and  2cm);

\draw (-12 ,0) arc (180:360:0.5cm and  2cm);
\draw (-12 ,0) arc (180:0:0.5cm and  2cm);
\draw[thick] (-11.5 , 2) to (-2.5, 2);
\draw[thick] (-11.5 , -2) to (-2.5, -2);

\draw  [dashed](-3  ,0) arc (180:360:0.5cm and  2cm);
\draw  [dashed](-3  ,0) arc (180:0:0.5cm and  2cm);
\draw  (-2.5 ,-2) arc (-90:90:0.5cm and  2cm);

\draw [thick, ->] (-10, -3.3) to (-5, -3.3); 
\node at (-4.5, -3.3) {$s$};

\draw  [thick, ->](-14 ,0) arc (180:360:0.5cm and  2cm);
\draw  [thick](-14 ,0) arc (180:30:0.5cm and  2cm);
\node at (-13, 0.4) {$t$};

\node at (-11.5, 2.7) {$\gamma_-$};
\node at (-2.5, 2.7) {$\gamma_+$};

  \end{scope}

 \end{tikzpicture}
%  \caption{ }
\end{subfigure}
\begin{subfigure}{0.45\textwidth}
  \centering
\begin{tikzpicture}[scale=0.35]
   \begin{scope}[>=latex]

%gray filling
\draw [gray!20, fill=gray!20] (3, 2)--(13, 2)--(13,-2)--(3,-2);
 
\node at (8, 2.7) {$\color{red}{\widehat L}$};
\node at (8, -1.3) {$\color{red}{\widehat L}$};
\node at (3.9, 0) {$x_-$};
\node at (13.9, 0) {$x_+$};

\draw (3, -2) to (3,2);
\draw[thick, red] (13, -2) to (3,-2);
\draw (13, -2) to (13,2);
\draw[thick, red] (13, 2) to (3,2);

\draw [thick, ->] (5.5, -3.3) to (10.5, -3.3); 
\node at (11, -3.3) {$s$};
\draw [thick, ->] (2, -2) to (2,2); 
\node at (1.5, 0) {$t$};

  \end{scope}

\end{tikzpicture}

 % \caption{ }
\end{subfigure}

	\caption{A Floer cylinder  (left) and strip (right)}
	\label{fig: cylandstr}
\end{figure}

\begin{proposition}
Let $\gamma_- \neq \gamma_+$. For generic $J_{S^1}$, the moduli space $\mathcal{M}(\gamma_-, \gamma_+, H_{S^1}, J_{S^1})$ is a smooth manifold of dimension $|\gamma_-| - |\gamma_+|-1$.
\end{proposition}
\begin{remark}\label{rem: noperiodcontactregion}
Since $H_{S^1}$ and $J_{S^1}$ are admissible, Floer trajectories must lie in a compact region in $\widehat W$ by a maximum principle. 
\end{remark}
The differential $\p: \CF_k(H_{S^1}, J_{S^1}) \rightarrow \CF_{k-1}(H_{S^1}, J_{S^1})$ is defined by counting rigid Floer trajectories between 1-periodic orbits as follows:
\begin{equation}\label{eq: SHdifferential}
\p (\gamma_-) = \sum_{\substack{\gamma_+ \in \mathcal{P}(H_{S^1}) \\ |\gamma_+| = k-1}} \#_2 \mathcal{M}(\gamma_-, \gamma_+, H_{S^1}, J_{S^1}) \gamma_+.
\end{equation}
The Floer--Gromov compactness and the gluing construction in Floer theory show that $\p^2 = 0$, and hence we obtain the Floer chain complex $(\CF_*(H_{S^1}, J_{S^1}), \p)$. The Floer homology $\HF_*(H_{S^1}, J_{S^1})$ of the pair $(H_{S^1}, J_{S^1})$ is defined by
$$
\HF_*(H_{S^1}, J_{S^1}) = \Ho_*(\CF_*(H_{S^1}, J_{S^1}), \p).
$$
\subsubsection{Symplectic homology} Standard continuation maps in Hamiltonian Floer homology define a direct system of Floer homology groups $\HF_*(H_{S^1}, J_{S^1})$ directed by increasing the slope $
\tau$ of Hamiltonains. See e.g.\ \cite[Section 4.4]{CiOa18}. The \emph{symplectic homology} of the Liouville domain $(W, \lda)$ is defined to be the direct limit
$$
\SH_*(W, \lda) = \varinjlim_{\tau \rightarrow \infty} \HF_*(H_{S^1}, J_{S^1}).
$$

%\subsubsection{Ring structure} The symplectic homology $\SH_*(W)$ admits a natural commutative ring structure with unity
%$$
%\SH_k(W) \times \SH_l(W) \rightarrow \SH_{k+l+n}(W)
%$$ 
%defined by counting Floer pair-of-pants. We refer the reader to [Ritter] for details. In our grading convention, the unit element $1_W$ lies in $\SH_n(W)$.   

\subsubsection{Action filtration} For an admissible Hamiltonian $H_{S^1}$ we have the associated    action functional $\mathcal{A}_{H_{S^1}} : \mathcal{L}\widehat W \rightarrow \R$ on the free loop space $ \mathcal{L}\widehat W$ of the completion $\widehat W$ given by
$$
\mathcal{A}_{H_{S^1}}(\gamma) = -\int_{S^1} \gamma^*\lda -\int_{0}^1 H_{S^1}(t, \gamma(t)) dt.
$$
We call the value $\mathcal{A}_{H_{S^1}}(\gamma)$ the \emph{action} of $\gamma$.  Since Floer trajectories decrease  action values, we obtain an action filtration on Floer chain complexes  by collecting generators of action less than $a \in \R$ 
$$
\CF_k^a(H_{S^1}, J_{S^1}) = \bigoplus_{\substack{\gamma \in \mathcal{P}(H_{S^1}) \\ |\gamma| = k \\ \mathcal{A}_{H_{{S^1}}}(\gamma) < a}} \Z_2 \langle \gamma \rangle.
$$
The corresponding filtered Floer homology is denoted by $\HF_*^a(H_{S^1}, J_{S^1})$, and taking the direct limit we define the \emph{filtered symplectic homology}
$$
\SH_*^a(W, \lda) = \varinjlim_{\tau \rightarrow \infty} \HF_*^a(H_{S^1}, J_{S^1}).
$$
\subsubsection{Tautological exact sequences}\label{sec: tauexseqSH} Let $a < b$. The action filtration on the chain complex $\CF_*(H_{S^1}, J_{S^1})$ induces the following natural short exact sequence of chain complexes:
$$
0 \rightarrow \CF_*^a(H_{S^1}, J_{S^1}) \rightarrow \CF_*^b(H_{S^1}, J_{S^1}) \rightarrow \CF_*^{[a, b)}(H_{S^1}, J_{S^1}) \rightarrow 0
$$
where $\CF_*^{[a, b)}(H_{S^1}, J_{S^1})$ is the chain complex defined to be the quotient 
$$
\CF_*^{[a, b)}(H_{S^1}, J_{S^1}) = \CF_*^b(H_{S^1}, J_{S^1}) / \CF_*^a(H_{S^1}, J_{S^1})
$$ 
with the induced differential. We obtain, passing to the direct limit, an associated long exact sequence in symplectic homology
\begin{equation}\label{eq: leqSH}
\rightarrow \SH_k^a(W) \rightarrow \SH_k^b(W) \rightarrow \SH_k^{[a, b)}(W)\rightarrow \SH_{k-1}^a(W) \rightarrow.
\end{equation}
In particular, due to the assumption that $H_{S^1}$ is $C^2$-small and Morse on $W$, if $\epsilon > 0$ sufficiently small, we have a canonical identification
\begin{equation}\label{eq: SHepsilonsmall}
\SH_k^{\epsilon}(W) \cong \Ho_{k+n}(W, \p W).
\end{equation} 
We then have the \emph{(filtered) tautological exact sequence} in symplectic homology
$$
\rightarrow \Ho_{k+n}(W, \p W) \rightarrow \SH_k^a(W) \rightarrow \SH_k^{[\epsilon, a)}(W)\rightarrow \Ho_{k+n-1}(W, \p W) \rightarrow.
$$
For each $a >0$, we shall denote the map from $\Ho_{k+n}(W, \p W)$ to $\SH_k^a(W)$ in the sequence by 
$$
j^a : \Ho_{k+n}(W, \p W) \rightarrow \SH_k^a(W).
$$
 
%The map $c_* : H_{k+n}(W, \p W) \rightarrow \SH_k(W)$ is of particular interest (here $b = \infty$), and has the following property. See [Ritter].
%\begin{proposition}
%The map $c_* : H_{2n}(W, \p W) \rightarrow \SH_n(W)$ in the tautological exact sequence sends the fundamental class $[W, \p W] \in H_{2n}(W, \p W)$ to the unit element $1_W \in \SH_n(W).$
%\end{proposition}

\subsection{Wrapped Floer homology} We shortly review a construction of wrapped Floer homology which is an open string analogue of symplectic homology. We refer to \cite{AboSei,KKK} for details. 

\subsubsection{Chain complex}\label{sec: wrappedchaincomplex} Let $L$ be an \emph{admissible} Lagrangian in a Liouville domain $(W, \lda)$, meaning that $L$ is a connected and exact Lagrangian which intersects the contact boundary $(\Sigma, \alpha)$ in a  Legendrian $\mathcal{L} := \p L = L \cap \Sigma$ and the Liouville vector field is tangent to $TL$ along the boundary. By attaching $[1, \infty) \times \mathcal{L}$ to  $L$ along the Legendrian boundary $\mathcal{L}$ we have a completed exact Lagrangian $\widehat L$ in the completion $(\widehat W, \widehat \lda)$. Roughly speaking, the wrapped Floer homology $\HW_*(L)$ is a version of Lagrangian Floer homology of $\widehat L$ in $\widehat W$. 

A time-independent Hamiltonian $H : \widehat W \rightarrow \R$ is called \emph{admissible} if every Hamiltonian 1-chord relative to $\widehat L$ is non-degenerate, $H$ is negative and $C^2$-small in the interior of $W \subset \widehat W$, and $H$ is linear at the end with respect to $r \in [1, \infty)$. We assume that the slope $\tau$ of $H$ is positive and is not equal to the length of a Reeb chord in $(\Sigma, \alpha, \mathcal{L})$. Recall that a \emph{Reeb chord} in $(\Sigma, \alpha, \mathcal{L})$ is an orbit $x: [0, T] \rightarrow \Sigma$ of the Reeb flow on $(\Sigma, \alpha)$ with $x(0), x(T) \in \mathcal{L}$. We call $T$ the \emph{length} of the Reeb chord $x$.

Denote the set of contractible, as an element of $\pi_1(\widehat W, \widehat L)$, Hamiltonian 1-chords by $\mathcal{P}_L(H)$. We associate the index $|x|=-\mu(x)-\frac{n}{2} \in \Z$ for each non-degenerate contractible 1-chord in $\mathcal{P}_L(H)$, where $\mu(x)$ is the Maslov index defined in \cite[Definition~2.3]{KKK}. Assume that $c_1(TW)=0$ and $\pi_1(L) = 0$ for well-definedness of~$\mu(x)$. 

%{\color{blue}
%\begin{remark}
%The applications in Section \ref{sec: Floerhomologycapacities} do not require the grading. The topological conditions on $W$ and $L$ are therefore not necessary for applications. See Section \ref{sec: nograding}.
%\end{remark}
%}

\begin{example}\label{ex: admLagconditions}
Consider the complex conjugation $\rho_0$ on the closed ball $B^{2n}$.  Its fixed point set $L = \Fix(\rho_0) = B^{2n} \cap \R^n$, called a real Lagrangian, defines an admissible Lagrangian in $(B^{2n}, \lda_0)$.
More generally, let $W$ be a starshaped domain in $\R^{2n}$ invariant under an exact anti-symplectic involution $\rho$ of $(\R^{2n},\lambda_0)$ i.e. $\rho^*\lda_0 = -\lda_0$. Then the fixed point set $L : = \Fix (\rho|_W)$ defines an admissible Lagrangian:
\begin{itemize}
\item From the classical Smith theory, we have $\dim \Ho_*(W;\Z_2) \geq \dim \Ho_*(L; \Z_2)$ and $\chi(W) = \chi(L) \mod 2$. It follows that $L$ is nonempty and connected. 
\item Since the Liouville flow on $(\R^{2n}, \lda_0)$ commutes with $\rho$ and flows radially from the origin, the real Lagrangian $L$ intersects the boundary $\p W$.
\item As in \cite[Lemma 3.1]{KKL18}, $L$ is an exact Lagrangian, the intersection $L \cap \p W$ is a Legendrian, and the Liouville vector field is tangent to $TL$ along the boundary. 
\end{itemize}
Note also that $c_1(TW) = 0$ for any starshaped domain $W$ whereas $\pi_1(L)$ is not necessarily a trivial group. If the anti-symplectic involution $\rho$ is \emph{linear} e.g. the complex conjugation, then $L$ is diffeomorphic to the closed ball $B^n$ and hence $\pi_1(L)=0$ in this case.
\end{example}

Let $J = \{J_t\}_{t \in [0, 1]}$ be an admissible time-dependent family of compatible almost complex structures. % See \cite{KKK} for the definition. 
For two distinct 1-chords $x_{\pm} \in \mathcal{P}_L(H)$, the moduli space $\mathcal{M}(x_-, x_+, H, J)$ of \emph{Floer strips} from $x_-$ to $x_+$, modulo the natural $\R$-action, is defined by 
\begin{align}\label{eq: HWmoduli}
\begin{split}
\mathcal{M}(x_-, x_+, H, J) = \{u: \R \times [0, 1] \rightarrow \widehat W \;|\; &\lim_{s \rightarrow \pm \infty} u(s, t) = x_{\pm}(t), \\ &(du - X_{H}\otimes dt )^{0,1} = 0, \; \\&u(s, 0), u(s, 1) \in \widehat L \} / \R.
\end{split}
\end{align}
See the right in Figure \ref{fig: cylandstr}. For generic $J$, the moduli space $\mathcal{M}(x_-, x_+, H, J)$ is a smooth manifold of dimension $|x_-| - |x_+|-1$.

The Floer chain complex for the pair $(H, J)$ is defined by
$$
\CF_k(H, J) = \bigoplus_{\substack{x \in \mathcal{P}_L(H) \\ |x| = k}} \Z_2 \langle  x \rangle
$$
equipped with the differential $\p : \CF_k(H, J) \rightarrow \CF_{k-1}(H, J)$ given by
\begin{equation}\label{eq: HWdiff}
\p (x_-) = \sum_{\substack{x_+ \in \mathcal{P}_L(H)\\ |x_+| = k-1}} \#_2 \mathcal{M}(x_-, x_+, H, J) x_+.
\end{equation}
We obtain the Floer homology group $\HF_*(H, J)$ as the homology of the chain complex $(\CF_*(H, J), \p)$, and by taking the direct limit as in the symplectic homology, we define the \emph{wrapped Floer homology} of the Lagrangian $L$ in $(W, \lda)$ by
$$
\HW_*(L) = \varinjlim_{\tau \rightarrow \infty}  \HF_*(H, J).
$$

%The wrapped Floer homology $\HW_*(L)$ admits a ring structure with unity, 
%$$
%\HW_k(L) \times \HW_l(L) \rightarrow \HW_{k+l}(L)
%$$
%which is non-necessarily commutative. The unit is denoted by $1_L \in \HW_0(L)$.

\subsubsection{Filtered wrapped Floer homology}\label{sec: tauexseqHW} Wrapped Floer homology shares many analogous properties with symplectic homology. In particular, we have a natural action filtration and tautological exact sequences. The action filtration on $\HW_*(L)$ is given by the action functional $\mathcal{A}_H : \mathcal{L}_L \widehat W \rightarrow \R$ on the free path space  $\mathcal{L}_{L}\widehat W$ of the completion $\widehat W$ relative to $\widehat L$, defined by
$$
\mathcal{A}_{H}(x) = -\int_{[0,1]} x^*\lda -\int_{0}^1 H(x(t)) dt + f_L(x(1)) - f_L(x(0)).
$$
Here $f_L \in C^{\infty}(\widehat{L}, \R)$ is a primitive of the form $\lda|_{\widehat L}$.
For $a \in \R$, we denote the filtered chain complex by $\CF_*^a(H, J)$ and the filtered wrapped Floer homology by $\HW_*^a(L)$. For $a < b$ a long exact sequence analogous to \eqref{eq: leqSH} is written as
\begin{equation}\label{eq: taulesHW}
\rightarrow \HW^a_k(L) \rightarrow \HW_k^b(L) \rightarrow \HW_k^{[a, b)}(L)\rightarrow \HW^a_{k-1}(L) \rightarrow.
\end{equation}
In particular, for $\epsilon >0$ sufficiently small so that
\begin{equation}\label{eq: HWepsilonsmall}
\HW_k^{\epsilon}(L) \cong \Ho_{k+n}(L, \p L)
\end{equation}
we have the tautological long exact sequence in wrapped Floer homology
$$
\rightarrow \Ho_{k+n}(L,\p L) \rightarrow \HW_k^a(L) \rightarrow \HW_k^{[\epsilon, a)}(L)\rightarrow \Ho_{k+n-1}(L, \p L) \rightarrow.
$$
For each $a > 0$, as in symplectic homology, we denote the map from $\Ho_{k+n}(L,\p L)$ to $\HW_k^a(L)$ in the sequence by
$$
j^a: \Ho_{k+n}(L,\p L) \rightarrow \HW_k^a(L).
$$
 
%\begin{proposition}
%The map $c_* : H_n(L, \p L) \rightarrow \HW_0(L)$ in the tautological exact sequence with $b = \infty$ sends the fundamental class $[L, \p L] \in H_n(L, \p L)$ to the unit element $1_L \in \HW_0(L).$
%\end{proposition}

\subsection{Closed-open maps} \label{sec: closedopenmaps}Closed-open maps are natural homomorphisms from symplectic homology to wrapped Floer homology.  In   Section \ref{sec: Floerhomologycapacities} we    use them to relate  symplectic capacities from  the two Floer homologies. In this section we shall briefly outline a construction of closed-open maps based on \cite{Abouzaid, Ganatra, RitterIvan}. See also \cite{Al08}.%[Section 5]{Al08}.

\subsubsection{Floer data} Closed-open maps are defined by counting curves in $\widehat W$ which we call \emph{Floer chimneys}.  The domain $\mathcal{T}$ of Floer chimneys is given by the closed unit disk $\D$ with an interior puncture and a boundary puncture,
$$
\mathcal{T}= (\D \setminus \{0, 1\}, i)
$$ 
where $i$ is the standard complex structure. See the left in Figure \ref{fig: chim}. We equip $\mathcal{T}$  a negative cylindrical end $\varepsilon_0: (-\infty, 0]  \times S^1 \rightarrow \mathcal{T}$ near $0$ and a positive strip-like end $\varepsilon_1: [0, \infty) \times  [0, 1]  \rightarrow \mathcal{T}$ near 1.

%\begin{figure}
   %\centering
%\begin{tikzpicture}[scale=0.4]

%\draw [fill=gray!20] (0,0) circle (3cm);
%inner circle
%\draw [dashed, black, fill=blue!30] (0,0) circle (0.5cm);
%brown
%\draw [dashed, fill=brown!30] (3, 0.5) arc (100:260: 0.5cm and 0.5 cm);
%outer red circle
%\draw (0,0) circle (3cm);

%holes
%\draw [fill=white](3,0) circle (0.1cm);
%\draw [fill=white](0,0) circle (0.1cm);

%\node at (0, -1) {$H_{S^1}$};
%\node at (2.2, 0.5) {$H$};
 %\end{tikzpicture}

	%\caption{Admissible Hamiltonians for Floer chimneys}
	%\label{fig: chimadmham}
%\end{figure}

A Floer data $(H_{\mathcal{T}}, J_{\mathcal{T}})$ for chimneys is given as follows. Let $H_{S^1}: S^1 \times \widehat W \rightarrow \R$ and $H: \widehat W \rightarrow \R$  be admissible Hamiltonians for symplectic homology and wrapped Floer homology, respectively, of the same slope $\tau$. A Hamiltonian $H_{\mathcal{T}}: \mathcal{T} \times \widehat W \rightarrow \R$ is called \emph{admissible} if  
\begin{itemize}
\item $H_{\mathcal{T}}(\varepsilon_0 (s, t), w) = H_{S^1}(t, w)$;
\item $H_{\mathcal{T}}(\varepsilon_1 (s, t), w) = H(w)$;
\item for each $z \in \mathcal{T}$, the Hamiltonian $H_{\mathcal{T}}(z, \cdot): \widehat W \rightarrow \R$ is admissible with slope~$\tau$ and is independent of $z$ at the end. We call $\tau$ the \emph{slope} of $H_{\mathcal{T}}$.
\end{itemize}
For admissible almost complex structures $J_{S_1}$ and $J$ as in Section \ref{sec: symplecticchaincomplex} and \ref{sec: wrappedchaincomplex}, we take an admissible $\mathcal{T}$-family of compatible almost complex structures $J_{\mathcal{T}}$ given in an analogous way to the Hamiltonian case so that $J_{\mathcal{T}} = J_{S^1}$ and $J_{\mathcal{T}} = J$ near the punctures.

\subsubsection{Floer chimneys} \label{sec: Floerchimneys}

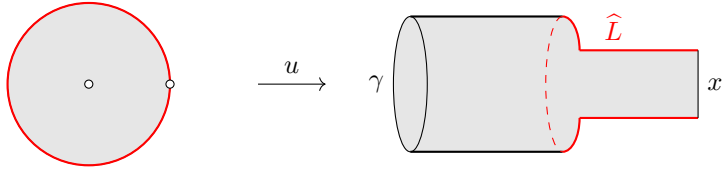
\begin{figure}[h]
   \centering
\begin{tikzpicture}[scale=0.45]

\begin{scope}[scale=0.8]
\draw [fill=gray!20] (0,0) circle (3cm);
%outer red circle
\draw[thick, red] (0,0) circle (3cm);

%holes
\draw [fill=white](3,0) circle (0.15cm);
\draw [fill=white](0,0) circle (0.15cm);
\end{scope}

\begin{scope}[xshift=14cm]
	%gray fill-in
\draw [gray!20, fill=gray!20] (-4.5,2)--(0,2)--(0.5,1)--(4,1)--(4,-1)--(0.5,-1)--(0,-2)--(-4.5,-2)--(-4.5,2);
\draw [gray!20, fill=gray!20] (-4.5,2) arc (90:270:0.5cm and 2cm);

%left ellipse
\draw (-4.5, 2) arc (90:270:0.5cm and  2cm);
%\draw [blue!30, fill=blue!30] (-4.53, 2)--(-4.05,2)--(-4.05,-2)--(-4.53,-2);

\draw  [gray!20,fill=gray!20](-4, 2) arc (90:270:0.5cm and  2cm);

\draw  (-4.5, -2) arc (-90:90:0.5cm and  2cm);

%left lines
\draw [thick] (-4.5, 2)--(0, 2);
\draw [thick] (-4.5, -2)--(0, -2);

%right ellipses
\draw  [thick, red, fill=gray!20](0.5, 1) arc (0:90:0.5cm and  1cm);
\draw  [thick, red, fill=gray!20](0.5 ,-1) arc (0:-90:0.5cm and  1cm);
\draw  [dashed, red](0 ,2) arc (90:270:0.5cm and  2cm);
%right lines
\draw [thick, red] (0.5, 1)--(4, 1);
\draw [thick, red] (0.5, -1)--(4, -1);
%right vertical
\draw (4,-1)--(4,1);

\draw[->] (-9,0) to (-7,0);

\node at (-8, 0.5) {$u$};
\node at (-5.5, 0) {$\gamma$};
\node at (4.5, 0) {$x$};
\node at (1.5, 1.7) {${\color{red} \widehat L}$};
\end{scope}

 \end{tikzpicture}
	\caption{A Floer chimney}
	\label{fig: chim}
\end{figure}

%\begin{figure}
%	\centering
%	\begin{overpic}[width=260pt,clip]{chim.pdf} 
%	\end{overpic}
%	\caption{Floer chimneys}
%	\label{fig: chim}
%\end{figure}

To write the Floer equation for chimneys, we fix a 1-form $\beta$ on $\mathcal{T}$ with the following properties.
\begin{itemize}
\item $d\beta \leq 0$ with respect to the fixed volume form on $\mathcal{T}$.
\item $\beta|_{\p \mathcal{T}} =0$, and $\beta|_{\nu(\p \mathcal{T})} = 0$ where $\nu(\p \mathcal{T})$ is a neighborhood of $\p \mathcal{T}$.
\item With respect to the coordinate charts $\varepsilon_0$ and $\varepsilon_1$, we set $\beta = dt$.
\end{itemize}

\begin{remark}
The conditions on $\beta$ guarantee that Floer chimneys stay in a compact region in $\widehat W$. One can show this using a convexity argument in \cite[Lemma 7.2]{AboSei}, which replaces the maximum principle.
\end{remark}
Take $\gamma \in \mathcal{P}(H_{S^1})$ and $x \in \mathcal{P}_L(H)$. A \emph{Floer chimney from $\gamma$ to $x$} is a map $u: \mathcal{T} \rightarrow \widehat W$ satisfying the following conditions, see Figure \ref{fig: chim}.
\begin{itemize}
\item (Floer equation) $u$ is a solution of the equation
$$
(du - X_{H_{\mathcal{T}}} \otimes  \beta)^{0, 1} = 0.
$$
\item (Asymptotic condition)
$$
\lim_{s \rightarrow -\infty} u(\varepsilon_0(s, t)) = \gamma(t), \quad \lim_{s \rightarrow \infty} u(\varepsilon_1(s, t)) = x(t).
$$
\item (Lagrangian boundary) $u(\p \mathcal{T}) \subset \widehat L$.
\end{itemize}
We denote the moduli space of Floer chimneys from $\gamma$ to $x$ by
\begin{equation}\label{eq: COmoduli}
\mathcal{M}(\gamma, x, H_{\mathcal{T}}, J_{\mathcal{T}}) =\{u: \mathcal{T} \rightarrow \widehat W \;|\; \text{$u$ is a Floer chimney from $\gamma$ to $x$}\}.
\end{equation}
\begin{proposition}[{See \cite[Lemma 5.3]{Abouzaid}}]
For generic $J_{\mathcal{T}}$, the moduli space $\mathcal{M}(\gamma, x, H_{\mathcal{T}}, J_{\mathcal{T}})$ is a smooth manifold of dimension
$$
\dim \mathcal{M}(\gamma, x, H_{\mathcal{T}}, J_{\mathcal{T}}) = |\gamma| - |x| -n.
$$
\end{proposition}
If $|\gamma| = |x| + n$, the moduli space is  compact and   zero dimensional. This allows us to define the map
$$
\mathcal{CO}: \CF_k(H_{S^1}, J_{S^1}) \rightarrow \CF_{k-n}(H, J)
$$
by counting rigid Floer chimneys
\begin{equation}\label{eq: COchainmap}
\mathcal{CO}(\gamma) = \sum_{\substack{x \in \mathcal{P}_L(H) \\ |x| = k-n}} \#_2 \mathcal{M}(\gamma, x, H_{\mathcal{T}}, J_{\mathcal{T}})  x.
\end{equation}
The codimension 1 boundary strata of the moduli space of Floer chimneys, described in \cite[Lemma 5.3]{Abouzaid}, shows that the map $\mathcal{CO}: \CF_k(H_{S^1}, J_{S^1}) \rightarrow \CF_{k-n}(H, J)$ is a chain map; see Figure \ref{fig: brochim}. We have the induced homomorphism on homology groups
$$
\mathcal{CO}: \HF_k(H_{S^1}, J_{S^1}) \rightarrow \HF_{k-n}(H, J).
$$
\begin{figure}
\centering
	\begin{tikzpicture}[scale=0.5]
	%left drawing
	\begin{scope}[xshift=-12cm]
			%gray fill-in
\draw [gray!20, fill=gray!20] (-4.5,2)--(0,2)--(0.5,1)--(4,1)--(4,-1)--(0.5,-1)--(0,-2)--(-4.5,-2)--(-4.5,2);
\draw [gray!20, fill=gray!20] (-4.5,2) arc (90:270:0.5cm and 2cm);
\draw [white, fill=white] (-2.95,2)--(-1.55,2)--(-1.55,-2)--(-2.95,-2)--(-2.95,2);
\draw  [gray!20,fill=gray!20](-3, -2) arc (-90:90:0.5cm and  2cm);

\draw  [gray!20,fill=gray!20](-1.5, 2) arc (90:270:0.5cm and  2cm);

%left ellipse
\draw (-4.5, 2) arc (90:270:0.5cm and  2cm);
\draw  [gray!20,fill=gray!20](-4, 2) arc (90:270:0.5cm and  2cm);
\draw  (-4.5, -2) arc (-90:90:0.5cm and  2cm);

%left ellipse 2
\draw [dashed] (-3, 2) arc (90:270:0.5cm and  2cm);
\draw (-3, -2) arc (-90:90:0.5cm and  2cm);

%left ellipse 3
\draw (-1.5, 2) arc (90:270:0.5cm and  2cm);
\draw  (-1.5, -2) arc (-90:90:0.5cm and  2cm);

%left lines
\draw [thick] (-4.5, 2)--(-3, 2);
\draw [thick] (-4.5, -2)--(-3, -2);

%left lines 2
\draw [thick] (-1.5, 2)--(0, 2);
\draw [thick] (-1.5, -2)--(0, -2);

%right ellipses
\draw  [thick, red, fill=gray!20](0.5, 1) arc (0:90:0.5cm and  1cm);
\draw  [thick, red, fill=gray!20](0.5 ,-1) arc (0:-90:0.5cm and  1cm);
\draw  [dashed, red](0 ,2) arc (90:270:0.5cm and  2cm);
%right lines
\draw [thick, red] (0.5, 1)--(4, 1);
\draw [thick, red] (0.5, -1)--(4, -1);
%right vertical
\draw (4,-1)--(4,1);

\node at (-5.5, 0.5) {$\gamma$};
\node at (4.5, 0.5) {$x$};
\node at (2, 1.5) {${\color{red} \widehat L}$};
\node at (-3.7, -2.7) {$\p$};
\node at (1, -2.7) {$\mathcal{CO}$};
	\end{scope}

%right drawing
%gray fill-in
\draw [gray!20, fill=gray!20] (-3.5,2)--(0,2)--(0.5,1)--(2,1)--(2,-1)--(0.5,-1)--(0,-2)--(-3.5,-2)--(-3.5,2);
\draw [gray!20, fill=gray!20] (-3.5,2) arc (90:270:0.5cm and 2cm);
\draw [gray!20, fill=gray!20] (3, 1)--(4.5, 1)--(4.5, -1)--(3, -1)--(3, 1);

%left ellipse
\draw (-3.5, 2) arc (90:270:0.5cm and  2cm);

%\draw  [gray!20,fill=gray!20](-4, 2) arc (90:270:0.5cm and  2cm);

\draw (-3.5, -2) arc (-90:90:0.5cm and  2cm);

%left lines
\draw [thick] (-3.5, 2)--(0, 2);
\draw [thick] (-3.5, -2)--(0, -2);

%right ellipses
\draw  [thick, red, fill=gray!20](0.5, 1) arc (0:90:0.5cm and  1cm);
\draw  [thick, red, fill=gray!20](0.5 ,-1) arc (0:-90:0.5cm and  1cm);
\draw  [dashed, red](0 ,2) arc (90:270:0.5cm and  2cm);

%right lines
\draw [thick, red] (0.5, 1)--(2, 1);
\draw [thick, red] (0.5, -1)--(2, -1);
%right vertical

\draw (2,-1)--(2,1);

%right lines/vertical 2
\draw [thick, red] (3, 1)--(4.5, 1);
\draw [thick, red] (3, -1)--(4.5, -1);
\draw (3,-1)--(3,1);
\draw (4.5,-1)--(4.5,1);

\node at (-4.5, 0.5) {$\gamma$};
\node at (5, 0.5) {$x$};
\node at (1.3, 1.5) {${\color{red} \widehat L}$};
\node at (3.7, 1.5) {${\color{red} \widehat L}$};
\node at (-1, -2.7) {$\mathcal{CO}$};
\node at (3.7, -2.7) {$\p$};

	\end{tikzpicture}
		\caption{Broken Floer chimneys showing $\p \circ \CO  = \CO \circ \p$}
	\label{fig: brochim}
\end{figure}
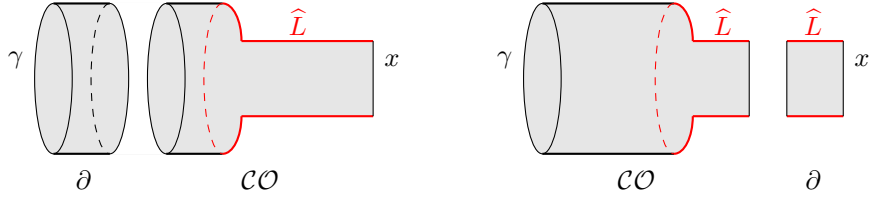\noindent
Taking homotopies of admissible Hamiltonians $H_{\mathcal{T}}$, a standard argument in Floer theory in \cite[Section 4.4]{CiOa18} allows us to pass it to the direct limit via continuation~maps
$$
\CO: \SH_k(W) \rightarrow \HW_{k-n}(L).
$$
We call this map the \emph{closed-open map} from symplectic homology to wrapped Floer homology.
\begin{remark}\label{rem: ringCO}
In \cite[Theorem 8.2]{RitterIvan}, it is shown that $\CO: \SH_*(W) \rightarrow \HW_*(L)$ is a unital ring homomorphism with respect to the standard ring structures described e.g. in \cite{Rit}.
\end{remark}

\subsubsection{Filtered closed-open maps} \label{sec: filteredclosedopenmaps} Closed-open maps respect the action filtrations. To see this, one introduces the \emph{topological energy} of Floer chimneys as follows:
$$
E(u) : = \int_{\mathcal{T}} u^*d\lda -u^*dH_{\mathcal{T}} \wedge \beta - u^*H_{\mathcal{T}} d\beta
$$
where $u \in \mathcal{M}(\gamma, x, H_{\mathcal{T}}, J_{\mathcal{T}})$ as in Section \ref{sec: Floerchimneys}. It is observed in \cite[Appendix B]{Abouzaid} that $E(u) \geq 0$, and a direct computation shows that 
$$
E(u) = \mathcal{A}_{H_{S^1}}(\gamma) - \mathcal{A}_H(x).
$$
In particular, Floer chimneys decrease action values. For each $a \in \R$,  we have \emph{filtered closed-open maps}
$$
\CO^a : \SH_k^a (W) \rightarrow \HW^a_{k-n}(L).
$$
The filtered closed-open maps are compatible with the tautological exact sequences \eqref{eq: leqSH} and \eqref{eq: taulesHW} in the following sense. This is also observed in \cite[Section 5.2.1]{BormanMclean} and essentially follows from \cite[Theorem 1.5]{Al08}.%Below $\epsilon >0$ is a small number in the sense of \eqref{eq: SHepsilonsmall} and \eqref{eq: HWepsilonsmall}.

\begin{theorem}\label{thm: closedopenmap}
The closed-open map
$$
\mathcal{CO}^a: \SH^a_k(W) \rightarrow \HW_{k-n}^a(L)
$$
for $a > 0$ fits into the following commutative diagram:

\begin{equation}\label{eq: diagclosedopen}
\begin{tikzcd}
\Ho_{k+n}(W, \p W) \arrow[r, "j^a"] \arrow[d, "i_!"] & \SH_k^a(W) \arrow[d, "\CO^a"] \\ 
\Ho_k(L, \p L)  \arrow[r, "j^a"] &  \HW_{k-n}^a(L)
\end{tikzcd}
\end{equation}
where the left vertical $i_!: \Ho_{k+n}(W, \p W) \rightarrow \Ho_k(L, \p L)$ is the natural map induced by the inclusion $i: L \rightarrow W$.
%In particular the map $\CO^{\epsilon}$ sends the fundamental class $[W, \p W] \in \Ho_{2n}(W, \p W)$ to the fundamental class $[L, \p L] \in H_n(L, \p L)$. 
\end{theorem}

\begin{proof}
Let  $\epsilon >0$ be a sufficiently small number such that \eqref{eq: SHepsilonsmall} and \eqref{eq: HWepsilonsmall} hold. The commutative diagram (on the right) in \cite[Theorem 1.5]{Al08} together with the compatibility of closed-open maps with the action filtrations yields the following commutative diagram.
\begin{equation}\label{eq: commcombinedocmaps}
\begin{tikzcd}
\Ho_{k+n}(W, \p W) \arrow[r, "\psi"] \arrow[d, "i_!"] & \SH^{\epsilon}_k(W) \arrow[r, "\iota_a"] \arrow[d, "\CO^{\epsilon}"]& \SH_k^a(W) \arrow[d, "\CO^a"] \\ 
\Ho_k(L, \p L)  \arrow[r, "\psi"] & \HW_{k-n}^{\epsilon}(L) \arrow[r, "\iota_a"]&  \HW_{k-n}^a(L)
\end{tikzcd}
\end{equation}
Here, the maps $\psi$, constructed in \cite[Section 15]{Rit}, are the analogues of the (relative) Piunikhin--Salamon--Schwarz isomorphism \cite{PSS}, and the maps $\iota_a$ are the natural inclusions from the respective tautological exact sequences \eqref{eq: leqSH} and \eqref{eq: taulesHW}. Indeed, the exactness of $L$ replaces the monotonicity assumption in \cite[Theorem 1.5]{Al08} from which we obtain the commutativity of the left square in \eqref{eq: commcombinedocmaps} modulo conventional differences between Floer homology and cohomology. The commutativity of the right square in \eqref{eq: commcombinedocmaps} is an immediate consequence of the fact that the closed-open map $\CO$ respects the action filtration on $\SH_*(W)$ and $\HW_*(L)$. As in \cite[Lemma 15.1]{Rit}, we know that $j^a = \iota_a \circ \psi$, and the commutative diagram \eqref{eq: diagclosedopen} in the assertion therefore follows from \eqref{eq: commcombinedocmaps}.
%{\color{blue}The commutative diagram is an immediate consequence of the fact that the closed-open map $\CO$ respects the action filtrations on $\SH_*(W)$ and $\HW_*(L)$.  The second assertion on $\CO^{\epsilon}$ follows from the fact that $\CO: \SH_*(W) \rightarrow \HW_*(L)$ is a unital ring homomorphism, see Remark \ref{rem: ringCO}, and the fact that the map $j = j^{\infty}: \Ho_{2n}(W, \p W) \rightarrow \SH_n(W)$ in the tautological exact sequence sends the fundamental class to the unit \cite{Rit}. Indeed, by the commutativity of the diagram with $a = \infty$, we have
%$$
%(j  \circ \CO^{\epsilon}) [W, \p W] = (\CO \circ j)[W, \p W] = \CO (1_W) = 1_L
%$$
%where $1_W \in \SH_n(W)$ and $1_L \in \HW_0(L)$ denote the respective unit elements. Since the unit $1_L$ is a non-zero element of $\HW_0(L)$, its preimage $\CO^{\epsilon}[W, \p W]$ under $j : \Ho_n(L, \p L) \rightarrow \HW_{0}(L)$ should also be a non-zero element of $\Ho_n(L, \p L)$. Recall that $L$ is assumed to be connected as in Section \ref{sec: wrappedchaincomplex}. It follows that $\Ho_n(L, \p L) \cong \Z_2$ and the unique non-zero element of $\Ho_n(W, \p W)$ is the fundamental class $[L, \p L]$. This concludes $\CO^{\epsilon}[W, \p W] = [L, \p L]$.}
\end{proof}

\subsubsection{Without absolute grading}\label{sec: nograding} Even though we have worked with the absolute $\Z$-grading on $\SH_*(W)$ and $\HW_*(L)$ for the sake of completeness, the Floer homologies and the filtered closed-open maps with Theorem \ref{thm: closedopenmap} readily work regardless of the grading. In fact the discussions in Section \ref{sec: Floerhomologycapacities} do not require the Floer homologies and the closed-open maps to be graded, and the topological assumptions $c_1(TW)=0$ and $\pi_1(L) = 0$ are therefore not necessary for our applications; see Example \ref{ex: admLagconditions}.

In this case, as a fairly standard way in Floer theory, we take the zero-dimensional component of the moduli spaces \eqref{eq: SHmoduli}, \eqref{eq: HWmoduli}, \eqref{eq: COmoduli} of Floer solutions instead of fixing the difference of the absolute gradings of asymptotes; the Fredholm index of underlying Fredholm problems determines the local dimension of the corresponding moduli spaces. Then we define the differentials \eqref{eq: SHdifferential}, \eqref{eq: HWdiff} and the chain map \eqref{eq: COchainmap} on ungraded Floer chain groups by counting elements of the zero-dimensional component of the respective moduli spaces. The analysis on Floer solutions and the action filtrations on Floer chain complexes are independent of the absolute grading. We therefore obtain the ungraded filtered closed-open map $\CO^{a}: \SH_*^a(W) \rightarrow \HW_*^a(L)$ with the commutative diagram:
$$
\begin{tikzcd}
\Ho_*(W, \p W) \arrow[r, "j^a"] \arrow[d, "i_!"] & \SH_*^a(W) \arrow[d, "\CO^a"] \\ 
\Ho_*(L, \p L)  \arrow[r, "j^a"] &  \HW_*^a(L)
\end{tikzcd}
$$

%\begin{remark}
%{\color{blue} Alternatively one can prove the second assertion in Theorem \ref{thm: closedopenmap} by identifying the map $\CO^{\epsilon}: \Ho_*(W, \p W) \rightarrow \Ho_*(L, \p L)$ with the natural map $\Ho^*(W) \rightarrow \Ho^*(L)$ induced by the inclusion $L \hookrightarrow M$.}
%\end{remark}

%To see that $\CO^{\epsilon}: H_{2n}(W, \p W) \rightarrow H_{n}(L, \p L)$ sends the fundamental class $[W, \p W]$ to the fundamental class $[L, \p L]$, consider part of the commutative diagram in Proposition \ref{prop: commdiagclosedopentau}:
%$$
%\begin{tikzcd}
%H_{2n}(W, \p W) \arrow[r, "c_*"] \arrow[d, "\CO^{\epsilon}"] & \SH_n(W) \arrow[d, "\CO"] \\ 
%H_n(L, \p L)  \arrow[r, "c_*"] &  \HW_{0}(L)
%\end{tikzcd}
%$$
%By [Smith--Ritter], we know $\CO(1_W) = 1_L$. As in [Ritter], the $c_*$ sends the fundamental class to the unit. By the commutativity it follows that $\CO^{\epsilon}([W, \p W]) = [L, \p L]$.

\section{Floer homology capacities} \label{sec: Floerhomologycapacities}

\subsection{SH capacity}\label{sec:SHcap} Let $(W, \lda)$ be a Liouville domain as in Section \ref{sec: SH}. We define the \emph{symplectic homology capacity} or shortly the \emph{$\SH$ capacity}  $c_{\SH}(W, \lda)$ of the domain $(W, \lda)$ by
$$
c_{\SH}(W) = c_{\SH}(W, \lda) = \inf\{a >0 \;|\; j^a [W, \p W] = 0\} \in [0, \infty]
$$
where the map $j^a : \Ho_*(W, \p W )\rightarrow \SH_*^a(W)$ is constructed in Section \ref{sec: tauexseqSH}. If $j^a [W, \p W] \neq 0$ for all $a >0 $, then we conventionally put $c_{\SH}(W) = \infty$. 

\begin{proposition}\label{prop: SHcapacityproperty}\
The $\SH$ capacity satisfies the following properties.
\begin{enumerate}
\item (Conformality) For a positive real number $r$, we have $$c_{\SH}(W, r \lda) = r c_{\SH}(W, \lda).$$
\item (Monotonicity) For a generalized Liouville embedding $(W_1, \lda_1) \hookrightarrow (W_2, \lda_2)$, we have
$$
c_{\SH}(W_1, \lda_1) \leq c_{\SH}(W_2, \lda_2).
$$
\item (Spectrality) If $c_{\SH}(W) < \infty$, there exists a periodic Reeb orbit $\gamma$ on the contact boundary $(\Sigma, \alpha)$ such that $$c_{\SH}(W) = \ell(\gamma)$$ where $\ell(\gamma)$ denotes the period of $\gamma$.
\end{enumerate}
\end{proposition}

\begin{remark}\label{rmk:genliuemd}
A symplectic embedding $\varphi : (W_1, \lda_1) \hookrightarrow (W_2, \lda_2)$ is called a \emph{generalized Liouville embedding} if $(\varphi^*\lda_2 - \lda_1)|_{\p W_1} = 0$ in $\Ho^1(\p W_1)$. In particular, if $W_1$ and $W_2$ are both starshaped domains in $\R^{2n}$, every symplectic embedding is a generalized Liouville embedding since $\Ho^1(S^{2n-1}) = 0$ for $n \geq 2$. 
\end{remark}

\begin{remark}
The $\SH$ capacity $c_{\SH}(W)$ is finite if and only if $\SH_*(W) = 0$.
\end{remark}

\begin{proof}[Proof of Proposition \ref{prop: SHcapacityproperty}]
For smooth convex domains in $\R^{2n}$, the above properties are presented e.g.\ in \cite[Section 2.4]{Irie}. For general Liouville domains, the monotonicity comes from the existence of a natural homomorphism $\SH_*^a(W_2) \rightarrow \SH_*^a(W_1)$, called a transfer map, in symplectic homology for generalized Liouville embeddings as in \cite[Theorem 1.24]{GH}. The spectrality follows from essentially the same argument as in \cite[Lemma 4.2]{GH}, using the relationship between action values of Hamiltonian 1-orbits of admissible Hamiltonians and Reeb orbits on the contact boundary; see \cite[Remark 5.6]{GH}.
\end{proof}

\subsection{HW capacity}\label{sec:HWcap} We can define an open string analogue of the SH capacity using wrapped Floer homology. Let $L$ be an admissible Lagrangian in a Liouville domain $(W, \lda)$. Recall that $L$ is assumed to be connected; Section \ref{sec: wrappedchaincomplex}. The \emph{wrapped Floer homology capacity} or shortly \emph{$\HW$ capacity}  is defined as
$$
c_{\HW}(W) = c_{\HW}(W, \lda, L) = \inf\{a >0 \;|\; j^a [L, \p L] = 0\}  \in [0, \infty]
$$
where the map $j^a : \Ho_*(L, \p L)\rightarrow \HW^a_*(L)$ is defined in Section \ref{sec: tauexseqHW}. We set $c_{\HW}(W) = \infty$ if $j^a [L, \p L] \neq 0$ for all $a > 0$. The following is completely analogous to that for the SH capacity  in Proposition \ref{prop: SHcapacityproperty}; we omit its proof.

\begin{proposition}\label{prop: propertiesofHWcapa}
The $\HW$ capacity satisfies the following properties.
\begin{enumerate}
\item (Conformality) For a positive real number $r$, we have $$c_{\HW}(W, r \lda, L) = r c_{\HW}(W, \lda, L).$$
\item (Monotonicity) For a generalized Liouville embedding $\varphi : (W_1, \lda_1) \rightarrow (W_2, \lda_2)$ with $\varphi(L_1) \subset L_2$, we have
$$
c_{\HW}(W_1) \leq c_{\HW}(W_2).
$$
\item (Spectrality) If $c_{\HW}(W) < \infty$, there exists a Reeb chord $x$ on the contact boundary $(\Sigma, \alpha, \mathcal{L})$ such that $$c_{\HW}(W) =\ell(x)$$ where $\ell(x)$ denotes the length of $x$.
\end{enumerate}
\end{proposition}

\begin{remark}
The SH capacity is also known as the \emph{Floer--Hofer--Wysocki capacity} defined in \cite{FHW}, and the HW capacity is referred to as \emph{Lagrangian Floer--Hofer--Wysocki capacity} in \cite{BormanMclean}.
\end{remark}

\begin{remark}\label{rem: SH=0=HW}
The $\HW$ capacity $c_{\HW}(W, \lda, L)$ is finite if and only if $\HW_*(L) =0$, which is in particular the case when $\SH_*(W) = 0$, see \cite[Theorem 10.6]{Rit}.
\end{remark}

\subsection{Proof of Theorem \ref{mainthm}}\label{sec:Upper} In this section, we give a proof of the estimate \eqref{main:estimate}. %following theorem using relations between $c_{\SH}$ and $c_{\HW}$.

%\begin{theorem}\label{thm: upperbound}
%Let $L$ be an admissible real Lagrangian in a convex body $K$ in $\R^{2n}$. The systole $\ell_{\min}$ and the symmetric systole $\ell_{\min}^{\rm sym}$ of the contact type boundary $(\p K, \alpha, \mathcal{L})$ satisfies
%$$
%\frac{\ell_{\min}^{\rm sym}(\p K, \alpha, \mathcal{L)}}{\ell_{\min} (\p K, \alpha)} \leq 2.
%$$
%\end{theorem}
%\begin{remark}
%The assumption in Theorem \ref{mainthm} for anti-symplectic involutions to be exact is not essential. For a given (possibly non-exact) anti-symplectic involution $\rho$ on $K$, we can find a Liouville form $\lda$ such that $d\lda = d\lda_0$ and $\rho$ is exact with respect to $\lda$. For example one takes $\lda := \frac{1}{2}(\lda_0 - \rho^*\lda_0)$. Note that the symmetric ratio can be rephrased in terms of the action of closed characteristics on the boundary $\p K$, and this is independent of the choice of Liouville forms. We can work with $\lda$ instead of $\lda_0$ in the proof of the estimate.
%\end{remark}

Let $K$ be a smooth compact convex domain in $\R^{2n}$ which is invariant under an anti-symplectic involution $\rho$ of $(\R^{2n},d\lambda_0)$, and the real Lagrangian $\Fix(\rho)$ intersects the boundary $\p K$. To apply our Floer setup, we choose a Liouville form $\lambda$ on $K$ with $d\lda = d\lda_0$ such that $\rho$ is an \emph{exact} anti-symplectic involution with respect to $\lda$ and the associated Liouville vector field is positively transverse along the boundary. For example one takes the average $\lambda:=\frac{1}{2}(\lambda_0-\rho^*\lambda_0)$:

\begin{lemma}\label{lem: ldaisalsofine}Let $(W,\lambda_0)$ be a Liouville domain with an anti-symplectic involution $\rho: W \rightarrow W$. Then $\lambda :=\frac{1}{2}(\lambda_0 - \rho^*\lambda_0)$ satisfies that $d\lda = d\lda_0$, $\rho^*\lambda=-\lambda$, and the corresponding Liouville vector field $X_\lambda$ of $\lambda$ is positively transverse along the boundary $\p W$.
\end{lemma}
\begin{proof}
It is immediate to see that $d\lambda=d\lambda_0$ and $\rho^*\lambda=-\lambda$. We claim that the Liouville vector field $X_\lambda$ of $\lambda$, defined by $d\lambda_0(X_\lambda,\cdot)=\lambda$, is positively transverse along the boundary $\p W$. Observe that $X_{\rho^*\lambda_0}=-\rho^*X_{\lambda_0}$. Indeed, for any vector $Y\in TW$,
\begin{align*}
	d\lambda_0 ( \rho^*X_{\lambda_0}, Y) &= \rho^*d\lambda_0(X_{\lambda_0}\circ \rho, \rho^*(Y\circ \rho)) \\
	&= -d\lambda_0(X_{\lambda_0}\circ \rho, \rho^*(Y\circ \rho)) \\
	&= -\lambda_0(\rho^*(Y\circ \rho)) \\
	&= -\rho^*\lambda_0(Y).
\end{align*}
Since the diffeomorphism $\rho: W\to W$ preserves the boundary $\p W$ and the interior of $W$, respectively, the push-forward of any outward normal vector along the boundary under $\rho$ is again an outward normal vector. From this fact, we deduce that the pull-back $\rho^*X_{\lambda_0}$ is positively transverse along the boundary. Therefore the convex sum
$$
X_\lambda=\frac{1}{2}(X_{\lambda_0}-X_{\rho^*\lambda_0})=\frac{1}{2}(X_{\lambda_0}+\rho^*X_{\lambda_0})
$$
is positively transverse along $\p W$ as well.
\end{proof}

Note that the systoles $\ell_{\min}(\p K)$ and $\ell_{\min}^{\rm sym}(\p K,\rho)$ defined in \eqref{eq: systoleforlda} do not depend on the choice of the Liouville form $\lambda$ as explained in Remark \ref{rem: chacfoli}, and we can work with $\lda$ instead of $\lda_0$. The triple $(K,\lambda,\rho)$ now fits our Floer setup as in Example \ref{ex: admLagconditions}. Note that we do not assume $\pi_1(L) = 0$; see Section \ref{sec: nograding}.

Abbreviate $\alpha = \lambda |_{\p K}$ and denote the restriction of $\rho$ to $\p K$ again by the same letter. First we relate the above Floer homology capacities with (symmetric) systoles. The following is a non-trivial fact relating the systole $\ell_{\min}(\p K)$ with the SH capacity, which is recently proved in \cite{AbonKang} and \cite{Irie}.

\begin{theorem}[Abbondandolo--Kang, Irie]\label{prop: AKIrie}
Let $K$ be a smooth convex body in $\R^{2n}$. Then the $\SH$ capacity of $K$ coincides with the systole of the contact boundary $( \p K, \alpha)$
$$
c_{\SH}(K) = \ell_{\rm min}(\p K ).
$$
\end{theorem} 

\begin{remark}\label{rem: comparionSHsystole}
The inequality $c_{\SH}(K)  \geq  \ell_{\rm min}(\p K )$ is obvious from the spectrality of $c_{\SH}(K)$ in Proposition \ref{prop: SHcapacityproperty}. There is a  starshaped and non-convex $K$ for which    the strict inequality $c_{\SH}(K)  >  \ell_{\rm min}(\p K )$ holds. See for example \cite[Section 3.5]{HZbook}.

\end{remark}
Since $\SH_*(K)=0$, it follows from Remark \ref{rem: SH=0=HW} that $\HW_*(L)=0$ as well. By the spectrality of $c_{\HW}$ in Proposition \ref{prop: propertiesofHWcapa}, there exists a symmetric periodic orbit on $(\p K,\rho)$.
In view of the one-to-one correspondence between symmetric periodic orbits and pairs of Reeb chords on the symmetric convex hypersurface $\p K$, the spectrality of $c_{\HW}$ % in Proposition~\ref{prop: propertiesofHWcapa} 
yields the following comparison.
\begin{proposition}\label{prop: symmsystoleHWcapacity}
The $\HW$ capacity of $(K, \rho)$ and the symmetric systole of $(\p K, \rho)$ satisfy
$$
\ell_{\min}^{\rm sym}(\p K, \rho) \leq 2c_{\HW}(K, \rho).
$$
\end{proposition} 

\begin{remark}
It should be possible to establish a real analogue of Theorem \ref{prop: AKIrie} for symmetric convex hypersurfaces, asserting that $2 c_{\HW}(K, \rho) = \ell_{\min}^{\rm sym}(\p K, \rho)$.
\end{remark}

Closed-open maps give the following relationship between the SH capacity and the HW capacity, which was also observed in \cite[(i) in Theorem 1.5]{BormanMclean}. We state it for general Liouville domains:

\begin{proposition}\label{prop: relationshhwcapacities}
For an admissible Lagrangian $L$ in a Liouville domain $(W, \lda)$, 
$$
c_{\HW}(W) \leq c_{\SH}(W).
$$
\end{proposition}

\begin{proof}
This is a direct consequence of Theorem \ref{thm: closedopenmap}. If $j^a[W, \p W] = 0$ in $\SH^{a}_*(W)$, it follows from the commutative diagram \eqref{eq: diagclosedopen}  that
$$
0 = (\CO^a \circ j^a) [W, \p W]  = (j^a \circ i_!)[W, \p W] = j^a[L, \p L]
$$
where the last equality holds because the natural map $i_!: \Ho_*(W, \p W) \rightarrow \Ho_*(L, \p  L)$ sends the fundamental class $[W, \p W]$ to the fundamental class $[L, \p L]$.  Therefore we have $a \geq c_{\HW}(W)$, and consequently we conclude $c_{\HW}(W) \leq c_{\SH}(W)$.
\end{proof}

We now obtain the desired estimate.

\begin{proof}[Proof of Theorem \ref{mainthm}]   Theorem \ref{prop: AKIrie} and Proposition \ref{prop: symmsystoleHWcapacity} tell  us that  
 $$
1 \leq  \mathfrak{R}(\p K, \rho) = \frac{\ell_{\min}^{\rm sym}(\p K, \rho) }{\ell_{\min} (\p K )} \leq  \frac{2 c_{\HW}(K, \rho)}{c_{\SH}(K)} .
$$
An application of Proposition  \ref{prop: relationshhwcapacities} to $ (K, \alpha, \rho)$ provides 
\[
\frac{ 2c_{\mathrm{HW}}(K, \rho)}{ c_{\mathrm{SH}}(K)} \leq 2,
\]
finishing the proof. 
\end{proof}

\section{Real symplectic capacities}\label{sec:examples}  
  Let $(M, \omega, \rho)$ be a real symplectic manifold, meaning that a symplectic manifold $(M, \omega)$ is  equipped with an anti-symplectic involution $\rho$, i.e.\ $\rho^* \omega = - \omega$. We always assume that $\mathrm{Fix}(\rho) \neq\emptyset$ so that it is a Lagrangian submanifold of $M$. A \emph{real symplectic embedding} $\Psi \colon (M_1, \omega_1,  \rho_1 ) \to (M_2, \omega_2,  \rho_2)$ between two  real symplectic manifolds  is an   embedding of $M_1$ into $M_2$ such that $\Psi^* \omega_2= \omega_1$ and  $\Psi^* \rho_2 = \rho_1$.  

\begin{definition}  A \textit{real symplectic capacity} is a function $c$ which assigns to a real symplectic manifold $(M, \omega, \rho)$    a number $c(M, \omega,  \rho) \in [0, + \infty]$ having the following properties:
\begin{itemize}
\item (Monotonicity) If real symplectic manifolds $ (M_1, \omega_1,  \rho_1 )$ and  $(M_2, \omega_2,  \rho_2)$  have the same dimension, and if there exists a real symplectic embedding $\Psi \colon (M_1, \omega_1,  \rho_1 ) \to (M_2, \omega_2,  \rho_2)$, then we have 
$c( M_1, \omega_1, \rho_1) \leq c(M_2, \omega_2, \rho_2)$;

\item (Conformality) $c(M,r \omega,    \rho) = r  c(M, \omega,  \rho)$ for all $ r>0$;

\item (Nontriviality) $0 < c(B^{2n}(1), \omega_0,  \rho_0)$ and $c(Z^{2n}(1), \omega_0,  \rho_0)< \infty$, where $B^{2n}(1) = \{z \in \C^{n} \;|\; ||z||^2 < 1\}$ and $Z^{2n}(1) = \{z \in \C^{n}\;|\; |z_1|^2 < 1\}$. Here $\omega_0 = d\lambda_0$ denotes the standard symplectic form on $\R^{2n}$, and $\rho_0$ denotes complex conjugation.

\end{itemize}
A real symplectic capacity  $c$ is said to be \textit{normalized} if 
\[
c(B^{2n}(1),  \omega_0,  \rho_0) = c(Z^{2n}(1), \omega_0,   \rho_0) =1.
\]

\end{definition}

\begin{remark}
The notion of real symplectic capacities was first introduced by Liu and Wang \cite{Liuwang}, where the authors referred to it as ``symmetrical'' symplectic capacities.  
\end{remark}

\begin{example}
  We provide several examples of real symplectic capacities. 
\begin{enumerate}
 \item[(i)] The \textit{real Gromov width} $c_B^{\mathrm{real}}( M, \omega, \rho) $  is defined as the supremum over all $r>0$ such that  $(B^{2n}(r),  \omega_0,   \rho_0)$  real symplectically embeds into $(M, \omega,  \rho)$. It is normalized and the smallest in the sense that if $c$ is a real symplectic capacity, then $c_B^{\mathrm{real}}(M , \omega, \rho) \leq c(M, \omega,  \rho)$ for all  real symplectic manifolds $(M, \omega, \rho)$.

\item[(ii)]  In \cite{Liuwang}, Liu and Wang constructed the    real Hofer--Zehnder capacity $c_{\mathrm{HZ}}^{\mathrm{real}}$, which is normalized,   by imitating the construction  of Hofer--Zehnder capacity  \cite{HZbook}.  Let $K \subset \R^{2n}$ be a compact convex domain invariant under a linear anti-symplectic involution $\rho$. It was shown in  \cite[Theorem 1.3]{JiLu20} that the real Hofer--Zehnder capacity of $(K, \rho)$ agrees with the symmetric systole, i.e.\ $c_{\mathrm{HZ}}^{\mathrm{real}}( K, \rho ) =  \ell_{\min}^{\mathrm{sym}}(\p K, \rho). $

\item[(iii)] Following the construction of the (first) Ekeland--Hofer capacity \cite{EHcapacity}, Jin and Lu defined the  real Ekeland--Hofer capacity   $c_{\mathrm{EH}}^{\mathrm{real}}(   \cdot  , \rho)$ for compact domains $K \subset \R^{2n}$   invariant under a fixed linear anti-symplectic involution $\rho$, see \cite{JiLu20}.   It is normalized.  Strictly speaking, it is not a real symplectic capacity as it is defined only for domain in $\R^{2n}$ and satisfies only restricted monotonicity: if $K_1 \subset K_2$ are compact domains in $\R^{2n}$ that are invariant under \emph{a fixed} linear anti-symplectic involution $\rho$, then we have $c_{\mathrm{EH}}^{\mathrm{real} }(K_1, \rho) \leq c_{\mathrm{EH}}^{\mathrm{real}}( K_2, \rho)$. Nonetheless, we  call it a real symplectic capacity.  For a compact convex domain $K \subset \R^{2n}$ invariant under a linear anti-symplectic involution $\rho$,   it agrees with the symmetric systole of $(\p K, \rho)$.  Consequently, for every symmetric convex domain $(K, \rho)$ with $\rho$ being linear,  the real Hofer--Zehnder capacity and the real Ekeland--Hofer capacity agree, see  \cite[Theorem 1.10]{JiLu20}.  
   
\item[(iv)] Let $(W, \lambda, \rho)$ be a real Liouville domain, i.e.\ $(W, \lambda)$ is a Liouville domain equipped with an exact anti-symplectic involution $\rho$. The wrapped Floer homology capacity  $c_{\mathrm{HW}}(W, \lambda, \rho)$, constructed using  wrapped Floer homology,   satisfies restricted monotonicity, meaning that if there exists a generalized real Liouville embedding from  $(W_1, \lambda_1, \rho_1)$ into  $(W_2, \lambda_2, \rho_2)$, then $c_{\mathrm{HW}}(W_1 , \lambda_1, \rho_1) \leq c_{\mathrm{HW}}(W_2, \lambda_2, \rho_2)$. See Section \ref{sec:HWcap} for the construction. Recall that a generalized   real symplectic embedding is a real symplectic embedding $\varphi \colon (W_1, \lambda_1, \rho_1) \to (W_2, \lambda_2, \rho_2)$ such that    $(\varphi^*\lda_2 - \lda_1)|_{\p W_1} = 0$ in $\Ho^1(\p W_1)$. In particular, if $W_1$ and $W_2$ are both starshaped domains in $\R^{2n}$, every  real symplectic embedding is a generalized  real Liouville embedding since $\Ho^1(S^{2n-1}) = 0$ for $n \geq 2$.  We expect that the argument of \cite{Irie} applies to all compact convex domains $K \subset \R^{2n}$ invariant under a linear anti-symplectic involution $\rho$, implying that $c_{\mathrm{HW} }(K, \rho) = \ell_{\min}^{\mathrm{sym}}(\p K, \rho)$.

\item[(v)]  Analogously to Gutt--Hutchings  \cite{GH}, we can construct, using positive equivariant wrapped Floer homology defined in \cite{KKK}, a sequence of real symplectic capacities $c_1 \leq c_2 \leq \cdots \leq \infty$ for real Liouville domains. They satisfy all the conditions for real symplectic capacities, but the monotonicity. Instead, they satisfy the restricted monotonicity as the wrapped Floer homology capacity. Using a Gysin-type exact sequence in wrapped Floer theory (see \cite[Proposition 3.27]{KKK} and \cite[Proposition 2.9]{BO_equiv}), it is not hard to see that  $c_1(W, \lambda,\rho) \leq c_{\mathrm{HW}}(W, \lambda, \rho)$ for every real Liouville domain $(W, \lambda, \rho)$.

 \end{enumerate}
 
   \end{example}

There is an old question about symplectic capacities asking if all normalized symplectic capacities agree on  compact convex domains in $\R^{2n}$, see \cite[Section 14.9, Problem 53]{MS17book} and \cite[Section 5]{Yaron}. We finish this article with the following related conjecture.  

\begin{conjecture}
For   convex domains in $\R^{2n}$ invariant under a fixed linear anti-symplectic involution, all normalized symplectic capacities and normalized real symplectic capacities are the same.
\end{conjecture}

\bibliographystyle{abbrv}
\bibliography{mybibfile}

\end{document}